\documentclass[11pt]{amsart}

\usepackage{amssymb,amsfonts,amscd,esvect}

\usepackage{geometry}
\usepackage[colorlinks,linktocpage]{hyperref}
\hypersetup{linkcolor=[rgb]{0,0,0.715}}
\hypersetup{citecolor=[rgb]{0,0.715,0}}

\newtheorem{thm}{Theorem}[section]
\newtheorem{cor}[thm]{Corollary}
\newtheorem{prop}[thm]{Proposition}
\newtheorem{lem}[thm]{Lemma}

\newtheorem*{thm1}{Theorem A}
\newtheorem*{cor2}{Corollary B}
\newtheorem*{prop3}{Proposition C}
\newtheorem*{thm4}{Theorem D}

\theoremstyle{definition}
\newtheorem{defn}[thm]{Definition}

\theoremstyle{remark}
\newtheorem{rem}[thm]{Remark}
\newtheorem{rems}[thm]{Remarks}

\newcommand{\step}{\mathrm{step}}
\newcommand{\dimH}{\dim_{\mathcal{H}}}

\makeatletter
\let\c@equation\c@thm
\makeatother
\numberwithin{equation}{section}

\title[]{Nonnegative Ricci curvature, nilpotency,\\ and Hausdorff dimension}
\author[]{Jiayin Pan}
\address[]{Department of Mathematics, University of California, Santa Cruz, California, US.}
\email{jpan53@ucsc.edu}

\begin{document}
	
	\begin{abstract}
		Let $M$ be an open (complete and non-compact) manifold with $\mathrm{Ric}\ge 0$ and escape rate not $1/2$. It is known that under these conditions, the fundamental group $\pi_1(M)$ has a finitely generated torsion-free nilpotent subgroup $\mathcal{N}$ of finite index, as long as $\pi_1(M)$ is an infinite group. We show that the nilpotency step of $\mathcal{N}$ must be reflected in the asymptotic geometry of the universal cover $\widetilde{M}$, in terms of the Hausdorff dimension of an isometric $\mathbb{R}$-orbit: there exist an asymptotic cone $(Y,y)$ of $\widetilde{M}$ and a closed $\mathbb{R}$-subgroup $L$ of the isometry group of $Y$ such that its orbit $Ly$ has Hausdorff dimension at least the nilpotency step of $\mathcal{N}$. This resolves a question raised by Wei and the author (see \cite[Remark 1.7]{PW_ex} and \cite[Conjecture 0.2]{Pan_cone}).
	\end{abstract}
	
	\maketitle
	
	\tableofcontents
	
	\parskip=5pt plus 1pt
	
	\section{Introduction}
	
	Collapsed Ricci limit spaces in general may admit isometric orbits whose Hausdorff dimension exceeds their topological dimension. The first examples with this feature are constructed by Wei and the author \cite{PW_ex} as the asymptotic cone of the universal cover of an open (complete and non-compact) manifold $M$ with $\mathrm{Ric}\ge 0$ and $\pi_1(M)=\mathbb{Z}$, resolving a well-known question by Cheeger-Colding on the Hausdorff dimension of singular sets in Ricci limit spaces \cite{CC97}. More precisely, the example is the equivariant Gromov-Hausdorff limit of
	$$(r_i^{-1}\widetilde{M},\tilde{p},\pi_1(M,p)=\mathbb{Z})\overset{GH}\longrightarrow (Y,y,L),$$
	where $r_i\to\infty$ and $(\widetilde{M},\tilde{p})$ is the universal cover of $(M,p)$. In the limit space, $L$ is a closed $\mathbb{R}$-subgroup of $\mathrm{Isom}(Y)$ and the orbit $Ly$ has Hausdorff dimension $1+\alpha$, where $\alpha\ge 0$ can be any large number by choosing a suitable metric and dimension of $M$.  
	In the same paper, we asked whether the (non-abelian) nilpotency of $\pi_1(M)$ implies the existence of some asymptotic $\mathbb{R}$-orbit of large Hausdorff dimension \cite[Remark 1.7]{PW_ex}. This question was later formalized by the author in \cite[Conjecture 0.2]{Pan_cone}, relating the nilpotency step to Hausdorff dimension. More precisely, for any open manifold $M$ with $\mathrm{Ric}\ge 0$ and escape rate not $1/2$, given that $\pi_1(M)$ contains a torsion-free nilpotent subgroup of nilpotency step $l$, is it true that $\dimH(Ly)\ge l$ for some asymptotic cone $(Y,y)$ of $\widetilde{M}$ and some closed $\mathbb{R}$-subgroup $L$ of $\mathrm{Isom}(Y)$?

	Before proceeding further, we give more background on this problem.
	
	Let $M$ be an open manifold with $\mathrm{Ric}\ge 0$. By the work of Kapovitch-Wilking \cite{KW}, $\pi_1(M)$ contains a nilpotent subgroup of index at most $C(n)$, a constant only depending on $n$. Also see \cite{Mil,Gro_poly}. In the other direction, by the work of Wei \cite{Wei} and Wilking \cite{Wilk}, it is known that any finitely generated virtually nilpotent can be realized as the fundamental group of some open manifold with $\mathrm{Ric}\ge 0$. This is distinct from open manifolds with nonnegative sectional curvature, whose fundamental groups are always virtually abelian \cite{CG_soul}. 
	
	Therefore, it is natural to investigate on what additional conditions $\pi_1(M)$ is virtually abelian for nonnegative Ricci curvature; or equivalently, we can ask how virtual abelianness or nilpotency of $\pi_1(M)$ is related to the geometry of $M$. The author has studied this question in \cite{Pan_es0,Pan_esgap,Pan_cone}. \cite{Pan_es0} introduced a geometric quantity, the escape rate $E(M,p)$, which measures how fast representing geodesic loops escape from bounded sets. The escape rate takes value within $[0,1/2]$. It is known that $E(M,p)<1/2$ implies the finite generation of $\pi_1(M)$ by Sormani's halfway lemma \cite{Sor}. In \cite{Pan_esgap}, we proved that if the escape rate of $M^n$ is smaller than some universal constant $\epsilon(n)$, then $\pi_1(M)$ is virtually abelian. In \cite{Pan_cone}, we proved that if the escape rate of $M$ is not $1/2$ and the universal cover is (metric) conic at infinity, then $\pi_1(M)$ is virtually abelian. The proofs in both \cite{Pan_esgap} and \cite{Pan_cone} relate the equivariant asymptotic geometry to the structure of $\pi_1(M)$.
	
	Given the above results, it is naturally to further study the equivariant asymptotic geometry without the smallness of escape rate and without the (metric) conic structure at infinity. Such understanding should have implications on virtual abelianness or nilpotency of $\pi_1(M)$. A close look at Wei's examples of open manifolds with $\mathrm{Ric}\ge 0$ and torsion-free nilpotent fundamental groups \cite{Wei} indicates that the nilpotency step is reflected in a lower bound of the Hausdorff dimension of isometric $\mathbb{R}$-orbits. See Section \ref{sec_exmp} for these motivating examples.
	
	This problem of nilpotency step and Hausdorff dimension is also related to the structure of Carnot groups. For a finitely generated virtually nilpotent group $\Gamma$, we define $\step(\Gamma)$ as the nilpotency step of a torsion-free nilpotent subgroup with finite index (see Definition \ref{def_virnilstep}). Any finite generating set $S$ of $\Gamma$ defines a word length metric $d_S$ on $\Gamma$. The asymptotic structure of $(\Gamma,d_S)$ was studied by Gromov \cite{Gro_poly} and Pansu \cite{Pansu}. For any sequence $r_i\to\infty$, Gromov-Hausdorff convergence holds:
	$$(r_i^{-1}\Gamma,e,d_S)\overset{GH}\longrightarrow (G,e,d).$$
	The unique limit space $(G,d)$ is a Carnot group, that is, a simply connected stratified nilpotent Lie group $G$ and a distance $d$ induced by a left-invariant subFinsler metric. There are two ways to detect the nilpotency step of $\Gamma$ from the limit $(G,e,d)$:\\
	(1) $\step(\Gamma)=\step(G)$;\\
	(2) $\step(\Gamma)=\max \{\dimH(L) | L \text{ is a one-parameter subgroup of $G$}\}$.\\
	In fact, $\step(G)=\step(\Gamma)=\dimH(L)$ holds for any one-parameter subgroup $L$ in $\zeta_{l-1}(G)$, the last nontrivial subgroup in the lower central series of $G$ (see Definition \ref{def_nilstep}). This structure result also applies to closed manifolds: for a closed Riemannian manifold $(M,g)$ with a virtually nilpotent fundamental group $\Gamma$, although $g$ cannot have nonnegative Ricci curvature when $\step(\Gamma)\ge 2$, the blow-down sequence of the universal cover $(r_i^{-1}\widetilde{M},\tilde{p},\tilde{g})$ converges in the Gromov-Hausdorff topology to a limit space $(G,e,d)$ as described above. Therefore, we can view the proposed problem as a natural extension of the Carnot group structure from closed manifolds to open ones.
	
	The main result of this paper confirms this conjecture about nilpotency step and Hausdorff dimension \cite[Conjecture 0.2]{Pan_cone}, thus resolves the question raised in \cite[Remark 1.7]{PW_ex}. We also give a description of any limit orbit $Gy$ as a simply connected nilpotent Lie group, extending the above-mentioned Carnot group structure.
	
\begin{thm1}\label{main}
	Let $(M,p)$ be an open manifold with $\mathrm{Ric}\ge 0$, an infinite $\Gamma:=\pi_1(M,p)$, and $E(M,p)\not=1/2$. We denote $\widetilde{M}$ the Riemannian universal cover of $M$. Then\\
	(1) For any equivariant asymptotic cone $(Y,y,G)$ of $(\widetilde{M},\Gamma)$, the limit orbit $Gy$ has a natural simply connected nilpotent Lie group structure with $\step(Gy)\le \step(\Gamma)$.\\
	(2) There exists an equivariant asymptotic cone $(Y,y,G)$ of $(\widetilde{M},\Gamma)$ and a closed $\mathbb{R}$-subgroup $L$ of $G$ such that $\dimH(Ly)\ge \step(\Gamma)$. 
\end{thm1}

We give some remarks on Theorem A.

\begin{rems}
(1) The condition $E(M,p)\not= 1/2$ assures that $\pi_1(M)$ is finitely generated. Thus $\step(\Gamma)$ is defined. $\step(\Gamma)$ is at least $1$ when $\pi_1(M)$ is infinite.\\
(2) The only known examples so far with $\mathrm{Ric}\ge 0$ and $E(M,p)=1/2$ are the ones with infinitely generated $\pi_1(M)$, recently constructed by Brue-Naber-Semola as counter-examples to the well-known Milnor conjecture \cite{BNS,BNS_dim6}. In their examples, the corresponding equivariant asymptotic geometry is very wild (see \cite[Section 2.2.4]{BNS}); in particular, in some of the asymptotic cone, the limit orbit $Gy$ can be compact, thus does not have the structure of a simply connected nilpotent Lie group. This demonstrates that the assumption $E(M,p)\not= 1/2$ in Theorem A is necessary for Theorem A(1) to hold.\\
(3) Under certain geometric conditions of $M$, $E(M,p)\not=1/2$ automatically holds. See, for example, \cite{NPZ}, where the theory developed in this paper is applied to manifolds with nonnegative Ricci curvature and linear volume growth.\\
(4) The inequalities $\step(Gy)\le \step(\Gamma)$ and $\dimH(Ly)\ge \step(\Gamma)$ in Theorem A extend the above-mentioned equality in the case of closed manifolds and asymptotic Carnot groups: $\step(G)=\step(\Gamma)=\dimH(L)$. We point out there are examples of open manifolds with strict inequalities $\step(Gy)< \step(\Gamma) < \dimH(Ly)$ (see Remark \ref{rem_abel_limit} or \cite[Appendix A]{Pan_cone}); in other words, both inequalities in Theorem A cannot be improved to equalities.
\end{rems}

%We emphasize that this nilpotent structure on $Gy$ could be abelian when $\step(\mathcal{N})>1$ . In other words, the (non-abelian) nilpotency structure of $\pi_1(M)$ may not be preserved in the algebraic aspect of asymptotic limits. However, according to Theorem A, it must be reflected in the metric aspects of asymptotic limits.

%The inequality in Theorem A is sharp. In fact, we have examples of open manifolds such that the equality holds on every asymptotic cone of $\widetilde{M}$ (see Section \ref{subsec_exmp_min}).
%The $\mathbb{R}$-subgroup $L$ in Theorem A indeed comes from $\langle \gamma \rangle$-action on $\widetilde{M}$, where $\gamma$ belongs to $\zeta_{l-1}(N)$. In terms of the asymptotic limit of $\mathcal{N}$-action, we can show that its asymptotic orbit $Gy$ always has a connected and simply connected nilpotent Lie group structure (Proposition \ref{topol_dim}).  
%It is unclear to the author whether the condition $E(M,p)\not=1/2$ can be replaced by the finite generation of $\pi_1(M)$, mainly due to the lack of examples with $\mathrm{Ric}\ge 0$, finitely generated $\pi_1(M)$, and $E(M,p)=1/2$. If such examples do exist, then we don't expect their equivariant asymptotic geometry to have good structures. Also see the recent examples by Bruè-Naber-Semola \cite{BNS} with infinitely generated $\pi_1(M)$ (thus $E(M,p)=1/2$); in particular, one may refer to , where the equivariant asymptotic geometry of their examples are described.

As applications of Theorem A, we can use the asymptotic geometry of $\widetilde{M}$ to control the nilpotency step of $\pi_1(M)$. For convenience, we write $\Omega(\widetilde{M})$ as the set of all asymptotic cones of $\widetilde{M}$ and define
$$\mathcal{D}_\infty(\widetilde{M})=\sup\{\dimH(Ly)|(Y,y)\in\Omega(\widetilde{M}), \text{ $L$ is a closed $\mathbb{R}$-subgroup of $\mathrm{Isom}(Y)$} \}.$$
Using this quantity, we can derive the corollary below from Theorem A(2).

\begin{cor2}\label{main_cor}
	Let $(M,p)$ be an open $n$-manifold with $\mathrm{Ric}\ge 0$ and $E(M,p)\not= 1/2$. Then $$\step(\pi_1(M))\le \mathcal{D}_\infty(\widetilde{M}).$$
	In particular, if $\mathcal{D}_\infty(\widetilde{M})<2$, then $\pi_1(M)$ is virtually abelian.
\end{cor2}

Corollary B generalizes the main result in \cite{Pan_cone}, where the universal cover $\widetilde{M}$ is assumed to be (metric) conic at infinity. When an asymptotic cone $(Y,y)$ is a metric cone with vertex $y$, the orbit $Ly$ always has Hausdorff dimension $1$ for any closed $\mathbb{R}$-subgroup $L$ of $\mathrm{Isom}(Y)$ (see proof of Corollary \ref{cor_cone} for details); in other words, $\mathcal{D}_\infty(\widetilde{M})=1$ if $\widetilde{M}$ is conic at infinity. Then it follows from Corollary B that $\pi_1(M)$ is virtually abelian.

Besides metric cones, Corollary B applies to other asymptotic cones that are not covered by previous results. For instance, following the methods in \cite{PW_ex} and \cite[Appendix A]{Pan_cone}, we can construct the Grushin halfspace below as the asymptotic cone of the universal cover of some open manifold with $\mathrm{Ric}\ge 0$ (also see \cite[Remark 3.9]{DHPW}, where the metric is clarified as a Grushin-type almost Riemannian metric). Given $0\le \alpha_1\le...\le \alpha_k$, we define an incomplete Riemannian metric $g$ on $\mathbb{R}^k \times (0,\infty)$ by
$$g=dr^2 + \sum_{j=1}^k r^{-2\alpha_j}dx_j^2.$$
By taking its metric completion, $g$ defines a distance $d$ on $Y=\mathbb{R}^k \times [0,\infty)$. We denote this Grushin halfspace $(Y,0,d)$ by $\mathbb{G}^+(\alpha_1,...,\alpha_k)$. Each $x_j$-curve through $0\in Y$ is the orbit of some isometric $\mathbb{R}$-action with Hausdorff dimension $1+\alpha_j$. Suppose that every asymptotic cone of $\widetilde{M}$ is isometric to a Grushin halfplane $\mathbb{G}^+(\alpha_1,..,\alpha_k)$ for some $0\le \alpha_1\le...\le \alpha_k$, then by Corollary B, $\step(\pi_1(M))\le 1+\alpha_k$
for some $Y\in \Omega(\widetilde{M})$; in particular, if $\alpha_k<1$ holds for all $Y\in \Omega(\widetilde{M})$, then $\pi_1(M)$ is virtually abelian.

Corollary B also extends the main result in \cite{Pan_esgap} about small escape rate and virtual abelianness. In fact, when $E(M,p)\le \epsilon$, we can show that the orbit of all asymptotic $\mathbb{R}$-orbits are close to $1$. See Proposition \ref{cor_small_escape} for the precise statement. 

\noindent\textbf{Outline of the proof.} To illustrate our approach to Theorem A, we further break the proof of Theorem A(2) into two parts, as Proposition C(1) and C(2) below. We write $\Omega(\widetilde{M},\langle\gamma \rangle)$ as the set of equivariant asymptotic cones of $(\widetilde{M},\langle\gamma \rangle)$, where $\gamma\in \pi_1(M,p)$ and $\langle \gamma \rangle$ is the subgroup generated by $\gamma$.

\begin{prop3}\label{main_Z}
	Let $(M,p)$ be an open $n$-manifold with $\mathrm{Ric}\ge 0$ and $E(M,p)\not= 1/2$. Let $\mathcal{N}$ be a torsion-free nilpotent subgroup of $\pi_1(M,p)$ with finite index and let $l$ be the nilpotency step of $\mathcal{N}$. Then the followings holds for any $\gamma\in \zeta_{l-1}(\mathcal{N})-\{\mathrm{id}\}$.\\
	(1) For every $(Y,y,H)\in\Omega(\widetilde{M},\langle\gamma \rangle)$, the orbit $Hy$ is homeomorphic to $\mathbb{R}$.\\
	(2) There exists $(Y,y,H)\in\Omega(\widetilde{M},\langle\gamma \rangle)$ such that $\dimH(Hy)\ge l$.
\end{prop3}

Below are the major steps in the proof of Theorem A(1) and Proposition C(1):\\
$E(M,p)\not=1/2$;\\
$\Rightarrow$ Equivariant GH distance gaps between different types of possible equivariant asymptotic cones of $(\widetilde{M},\mathcal{N})$ (Propositions \ref{eGH_gap_1} and \ref{eGH_gap_2});\\
$\Rightarrow$ Theorem A(1) and a uniform dimension among all orbits $Gy$ (Proposition \ref{topol_dim});\\
$\Rightarrow$ Proposition C(1).

While some intermediate results (for example, Propositions \ref{topol_dim} and C(1)) have their counterparts in \cite{Pan_cone}, where $\widetilde{M}$ is assumed to be conic at infinity, many new techniques are developed in this paper to study the equivariant asymptotic geometry without the metric cone structures. Compared to the conic at infinity case, the new difficulties are mainly due to the lack of structure results on singular sets with large Hausdorff dimension. When an asymptotic cone $(Y,y)$ is a metric cone with vertex $y$, it follows from Cheeger-Colding splitting theorem \cite{CC96} and the cone structure that the orbit $Gy$ of an isometric $G$-action must stay in an Euclidean factor of $Y$. This fact about metric cones is used extensively in \cite{Pan_cone}. In contrast, in the proof of Theorem A here, we are essentially considering orbits of large Hausdorff dimension in collapsed Ricci limits. Such examples are first constructed in \cite{PW_ex} and no general structure results are known so far.

Among the new techniques, we highlight the proof of distance gaps (Section \ref{sec_egh_gap}). We develop convergence of tunnels (continuous curves in the orbit) and apply the large fiber lemma (Lemma \ref{large_fiber}) from topological dimension theory to derive equivariant GH distance gaps between equivariant asymptotic cones of different types. We shall describe more about this novel method in Section \ref{subsec_egh_idea}.

Proposition \ref{topol_dim} depends on the above mentioned distance gaps and a critical rescaling argument. This kind of arguments was first developed in \cite{Pan_eu} and applied in different contexts to study the equivariant asymptotic geometry \cite{Pan_al_stable,Pan_es0,Pan_esgap,Pan_cone}. 

The proof of Proposition C(2) relies on many structure results of spaces in $\Omega(\widetilde{M},\langle \gamma\rangle)$, including Proposition C(1). Below we give some indications how Hausdorff dimension of $Hy$ and nilpotency step are related in the proof of Proposition C(2).

For $\mathcal{N}$ and $\gamma$ in Proposition C, it is known that any word length metric $d_S$ on $\mathcal{N}$ satisfies
$$C_1\cdot b^{1/l} \le d_S(\gamma^b \tilde{p},\tilde{p})\le C_2\cdot b^{1/l}$$
In fact, this two-side inequality implies $\dimH(Ly)=l$ in a Carnot group $G$, where $L$ is the asymptotic limit of $\langle \gamma \rangle$. For an isometric $\mathcal{N}$-orbit on $\widetilde{M}$, the nilpotency step only offers an upper bound:
$$d(\gamma^b \tilde{p},\tilde{p})\le C\cdot b^{1/l}$$
for all $b\in\mathbb{Z}_+$ (see Corollary \ref{orbit_length_upper}).

We approach Proposition C(2) by contradiction. We define
$$\mathcal{D}:=\sup\{\dimH(Hy)|(Y,y,H)\in\Omega(\widetilde{M},\langle\gamma \rangle)\}$$
and seek a contradiction if $\mathcal{D}<l$. We remark that $\mathcal{D}$ as a supremum indeed can be obtained (Proposition \ref{dim_sup}). The idea is to relate the Hausdorff dimension to a lower bound for orbit length: for each $s>\mathcal{D}$, there is a constant $C'$ such that 
$$d(\gamma^b \tilde{p},\tilde{p})\ge C'\cdot b^{1/s}$$
for all $b\in \mathbb{Z}_+$ large (Proposition \ref{orbit_length_lower}). Then a contradiction to the orbit length upper bound would arise if $\mathcal{D}<l$. 

To close the introduction, we point out that the proof of Theorem A and Proposition C naturally extend to nilpotent isometric actions on open manifolds with $\mathrm{Ric}\ge 0$ as well.

\begin{thm4}
	Let $M$ be an open $n$-manifold with $\mathrm{Ric}\ge 0$ and let $\mathcal{N}$ be a simply connected nilpotent Lie group with nilpotency step $l$. Suppose that $\mathcal{N}$ acts effectively and isometrically on $M$ and $(M,p,G)$ has escape rate less than $1/2$ (see Definition \ref{defn_es_rate_general}). Then\\
	(1) For every $(Y,y,G)\in \Omega(M,\mathcal{N})$, the limit orbit $Gy$ has a natural simply connected nilpotent Lie group structure with $\step(Gy)\le l$. \\
	(2) For every $(Y,y,H)\in\Omega(\widetilde{M},\langle\gamma \rangle)$, where $\gamma \in \zeta_{l-1}(\mathcal{N})-\{\mathrm{id}\}$, the limit orbit $Hy$ is homeomorphic to $\mathbb{R}$.\\
	(3) There is $(Y,y,H)\in\Omega(\widetilde{M},\langle\gamma \rangle)$, where $\gamma \in \zeta_{l-1}(\mathcal{N})-\{\mathrm{id}\}$, and a closed $\mathbb{R}$-subgroup $L$ of $H$ such that $\dimH(Ly)\ge l$.
\end{thm4}

\noindent\textbf{Acknowledgments.} The author is partially supported by National Science Foundation DMS-2304698 and Simons Foundation Travel Support for Mathematicians. A part of the paper was drafted when the author was supported by Fields Postdoctoral Fellowship from the Fields Institute. The author would like to thank Vitali Kapovitch and Guofang Wei for helpful discussions during the preparation of this paper.

\section{Preliminaries}

\subsection{Equivariant Gromov-Hausdorff convergence}

Throughout the paper, we use a tuple $(Y,y,G)$ to denote a pointed complete length metric space $(Y,y)$ with a closed subgroup $G$ of the isometry group $\mathrm{Isom}(Y)$. We recall the basics of (pointed) equivariant Gromov-Hausdorff convergence from \cite{Fu,FY}. For $R>0$, we put
$$G(R)=\{g\in G\ |\ d(gy,y)\le R \}.$$

\begin{defn}\label{defn_FY}\cite{Fu,FY} Let $(Y,y,G)$ and $(Z,z,H)$ be two spaces. We say that $$d_{GH}((Y,y,G),(Z,z,H))\le \epsilon,$$
where $\epsilon>0$, if there are $\epsilon$-approximation maps $(f,\psi,\phi)$, that is,
$$f:B_{1/\epsilon}(y)\to Z,\quad  \psi: G(1/\epsilon) \to H(1/\epsilon), \quad \phi:H(1/\epsilon) \to  G(1/\epsilon) $$
with the following properties:\\
(1) $f(y)=z$;\\
(2) the $\epsilon$-neighborhood $f(B_{1/\epsilon}(y))$ contains $B_{1/\epsilon}(z)$;\\
(3) $|d(f(x),f(x'))-d(x,x')|\le\epsilon$ for all $x,x'\in B_{1/\epsilon}(y)$;\\
(4) if $g\in G(1/\epsilon)$ and $x,gx\in B_{1/\epsilon}(y)$, then $d(f(gx),\psi(g)f(x))\le\epsilon$;\\
(5) if $h\in H(1/\epsilon)$ and $x,\phi(h)x\in B_{1/\epsilon}(y)$, then $d(f(\phi(h)x),h f(x))\le\epsilon$.
\end{defn}

\begin{thm}\label{eqGH_FY}\cite{Fu,FY}
	Let 
	$$(Y_i,y_i)\overset{GH}\longrightarrow (Z,z)$$
	be a Gromov-Hausdorff convergence sequence and let $G_i$ be closed subgroups of $\mathrm{Isom}(Y_i)$. Then the followings hold.\\
	(1) After passing to a subsequence, we have an equivariant Gromov-Hausdorff convergence sequence
	$$(Y_i,y_i,G_i)\overset{GH} \longrightarrow (Z,z,H),$$
	where $H$ is closed subgroup of $\mathrm{Isom}(Z)$.\\
	(2) The sequence of quotient spaces converges
	$$(Y_i/G_i,\bar{y}_i)\overset{GH}\longrightarrow (Z/H,\bar{z}).$$
\end{thm}

Later in Sections \ref{subsec_one_para_orbit} and \ref{subsec_pf_c1}, we shall use the convergence of symmetric subsets $S_i\subseteq G_i$. Recall that a subset $S$ in a group $G$ is \textit{symmetric}, if $e\in S$ and $g^{-1}\in S$ for every $g\in S$.

\begin{defn}\label{def_conv_symsubset}
	Let 
	$$(Y_i,y_i,G_i)\overset{GH}\longrightarrow (Z,z,H)$$
	be an equivariant Gromov-Hausdorff convergent sequence and let $S_i$ be a sequence of closed symmetric subsets in $G_i$. We say that
	$$(Y_i,y_i,S_i)\overset{GH}\longrightarrow (Z,z,S),$$
	for a closed symmetric subset $S$ of $H$, if\\
	(1) for every $h\in S$, there is a sequence $g_i\in S_i$ converging to $h$,\\
	(2) every sub-convergent sequence $g_i \in S_i$ has the limit $h$ in $S$.
\end{defn}

A corresponding precompactness result follows directly from the proof of \cite[Proposition 3.6]{FY}.

\begin{prop}
	Let 
	$$(Y_i,y_i,G_i)\overset{GH}\longrightarrow (Z,z,H)$$
	be an equivariant Gromov-Hausdorff convergent sequence and let $S_i$ be a sequence of closed symmetric subsets in $G_i$. Then after passing to a subsequence, we have convergence
	$$(Y_i,y_i,S_i)\overset{GH}\longrightarrow (Z,z,S),$$
	 where $S$ is a closed symmetric subset of $H$.
\end{prop}

Let $M$ be an open manifold with $\mathrm{Ric}\ge 0$. Let $\widetilde{M}$ be its Riemannian universal cover and let $\Gamma=\pi_1(M,p)$. For a sequence $r_i\to\infty$, after passing to a subsequence, we have convergence
$$(r_i^{-1}\widetilde{M},p,\Gamma)\overset{GH}\longrightarrow (Y,y,G).$$
The limit space $(Y,y,G)$ is called an \textit{equivariant asymptotic cone} of $(\widetilde{M},\Gamma)$. By the work of Colding-Naber \cite{CN}, $G$ is a Lie group.

In general, the above limit space $(Y,y,G)$ is not unique and may depend on the choice of $r_i\to\infty$. We denote $\Omega(\widetilde{M},\Gamma)$ as the set of all equivariant asymptotic cones of $(\widetilde{M},\Gamma)$.

\begin{prop}\label{cpt_cnt}
	The set $\Omega(\widetilde{M},\Gamma)$ is compact and connected in the pointed equivariant Gromov-Hausdorff topology.
\end{prop}

See \cite[Proposition 2.1]{Pan_eu} for a proof of Proposition \ref{cpt_cnt}.

\subsection{Nilpotent groups}

We recall some basics results about nilpotent groups.

\begin{defn}\label{def_nilstep}
A group $\mathcal{N}$ is \textit{nilpotent}, if the following central series terminates 
$$\mathcal{N}=\zeta_0(\mathcal{N})\triangleright \zeta_1(\mathcal{N}) \triangleright ... \triangleright \zeta_l(\mathcal{N})=\{\mathrm{id}\},$$
where $\zeta_{j+1}(\mathcal{N})=[\mathcal{N},\zeta_j(\mathcal{N})]$. We define the \textit{nilpotency step or class} of $\mathcal{N}$ as the smallest integer $l$ such that $\zeta_l(\mathcal{N})=\{\mathrm{id}\}$.
\end{defn}

It is clear that $\zeta_{l-1}(\mathcal{N})$ is contained in $Z(\mathcal{N})$, the center of $\mathcal{N}$.
  
	%(2) We define the rank of $\mathcal{N}$ by
	%$$\rank(\mathcal{N})=\sum_j \rank(\zeta_j(\mathcal{N})/\zeta_{j+1}(\mathcal{N})).$$  

The following result is standard for finitely generated nilpotent groups; see \cite{KM}.

\begin{prop}\label{torsion_free_index}
	Let $\Gamma$ be a finitely generated nilpotent group. Then\\
	(1) $\Gamma$ has a torsion-free nilpotent subgroup $\mathcal{N}$ of finite index;\\
	(2) if $\mathcal{N'}$ is another torsion-free nilpotent subgroup of finite index, then $\step(\mathcal{N})=\step(\mathcal{N'})$.
\end{prop}

By Proposition \ref{torsion_free_index}, we can naturally define nilpotency steps for finitely generated virtually nilpotent groups.

\begin{defn}\label{def_virnilstep}
	Let $\Gamma$ be a finitely generated virtually nilpotent group. We define
	$$\step(\Gamma):=\text{nilpotency step of } \mathcal{N},$$
	where $\mathcal{N}$ is a torsion-free nilpotent subgroup of $\Gamma$ with finite index.
\end{defn}

Note that under this definition, any finite nilpotent group $\Gamma$ has $\step(\Gamma)=0$.

The growth rate of finitely generated nilpotent groups are well understood by the work of Bass \cite{Bass} and Guivarc'h \cite{Guiv}. In particular, the following word length estimate holds.

\begin{prop}\label{word_length_estimate}\cite{Bass,Guiv}
	Let $\mathcal{N}$ be a finitely generated torsion-free nilpotent group of nilpotency step $l$ and let $\gamma$ be an element in $\zeta_{l-1}(\mathcal{N})-\{\mathrm{id}\}$. Let $S$ be a finite generating set of $\mathcal{N}$ and let $d_S$ be the corresponding word length metric on $\mathcal{N}$. Then there are positive constants $C_1$ and $C_2$ such that
	$$C_1\cdot  b^{1/l} \le d_S(\gamma^b,e)\le C_2\cdot  b^{1/l}$$
	for all $b\in \mathbb{Z}_+$.
\end{prop}

Proposition \ref{word_length_estimate} immediately implies an upper bound on the orbit length for $\mathcal{N}$-action on a metric space.

\begin{cor}\label{orbit_length_upper}
	Let $\mathcal{N}$ and $\gamma$ as in Proposition \ref{word_length_estimate}. Suppose that $\mathcal{N}$ acts freely and discretely on a metric space $(X,d)$ by isometries. Then for every $x\in X$, there is a constant $C$ such that 
	$$d(\gamma^b x,x)\le C\cdot b^{1/l}$$
	for all $b\in \mathbb{Z}_+$.
\end{cor}

\begin{proof}
	Let $S$ be a finite generating set of $\mathcal{N}$ and let
	$$R:=\max_{g\in S} d(g x,x).$$
	By Proposition \ref{word_length_estimate}, we can express $\gamma^b$ by at most $C_2b^{1/l}$ many elements in $S$. Thus by triangle inequality, it is clear that
	$$d(\gamma^b x,x)\le R\cdot C_2b^{1/l}.$$
\end{proof}

In general, the lower bound in Proposition \ref{word_length_estimate} cannot be transferred to orbit length on a non-compact metric space $X$.

Lastly, below are some results about nilpotent Lie groups. See \cite[Section 11.2]{HN} for references.

\begin{lem}\label{max_torus}
	Let $G$ be a connected nilpotent Lie group. Then any maximal torus of $G$ is central in $G$; in particular, $G$ has a unique maximal torus.
\end{lem}

\begin{lem}\label{nil_exp}
	Let $G$ be a connected nilpotent Lie group. Then the exponential map $\exp:\mathrm{Lie}(G)\to G$ is a smooth covering map. If in addition that $G$ is simply connected, then $\exp$ is a diffeomorphism.
\end{lem}

For an element $g\not= e$ in a connected and simply connected nilpotent Lie group $G$, according to Lemma \ref{nil_exp}, $g$ uniquely determines a parameter subgroup $\{\exp(tv)| t\in\mathbb{R} \}$ of $G$, where $v \in \mathrm{Lie}(G)$ such that $\exp(v)=g$. %For convenience, we usually denote elements in this one parameter subgroup as $tg$, where $t\in\mathbb{R}$. 

\begin{lem}\label{nil_G0_commute_cpt}
	Let $G$ be nilpotent Lie group. Then $G_0$ commutes with any compact subgroup $K$ of $G$.
\end{lem}

\begin{proof}
	For readers' convenience, we give a proof of this standard result in nilpotent Lie groups. Let $\mathfrak{g}$ be the Lie algebra of $G$ and let 
	$$\mathrm{Ad}:G\to \mathrm{Aut}(\mathfrak{g})$$
	be the adjoint representation. For any $g\in G$ and any $X\in \mathfrak{g}$, by nilpotency we have
	$$[g,[...,[g,\exp(tX)]]...]=\mathrm{id}.$$
	Differentiating both sides, we have $(\mathrm{Ad}(g)-I)^k=0$, that is, $\mathrm{Ad}(g)$ is unipotent. Hence under a suitable basis, $\mathrm{Ad}(g)$ is upper triangular with eigenvalues $1$. In particular, for any $\mathrm{Ad}(g)\not= I$, it must generate a non-compact subgroup. Therefore, if $g\in K$, a compact subgroup of $G$, then $\mathrm{Ad}(g)=I$. It follows that $g\in K$ commutes with any element $\exp(tX)\in G_0$. 	
\end{proof}

\subsection{Escape rate}

\cite{Pan_es0} introduces the notion of escape rate by comparing the size of representing geodesic loops to their lengths. Let $M$ be an open manifold with an infinite $\pi_1(M)$. For every element $\gamma\in \pi_1(M,p)$, we draw a representing loop $c_\gamma$ at $p$ such that it has the minimal length among all loops at $p$ in the homotopy class $\gamma$. $c_\gamma$ is always a geodesic loop at $p$. We denote
$$|\gamma|=d(\gamma \tilde{p},\tilde{p})=\mathrm{length}(c_\gamma).$$
The escape rate of $(M,p)$ is defined as
$$E(M,p)=\limsup_{|\gamma|\to\infty} \dfrac{d_H(p,c_\gamma)}{|\gamma|},$$
where $d_H(p,c_\gamma)$ is the Hausdorff distance between $c_\gamma$ and $p$ (in other words, the smallest radius $R$ such that the closed ball $\overline{B_R(p)}$ covers $c_\gamma$.) If $\pi_1(M)$ is finite, then we set $E(M,p)=0$ as a convention.

It is clear that $E(M,p)$ takes its value in $[0,1/2]$. In \cite{Sor}, Sormani proved that if $\pi_1(M,p)$ is not finitely generated, then there is a sequence of element $\gamma_i$ (the short generators) such that their representing loops $c_{\gamma_i}$ are minimal up to halfway. Therefore, in the terminology of escape rate, we have

\begin{thm}\cite{Sor}\label{escape_onehalf}
	Let $M$ be an open manifold with $E(M,p)<1/2$, then $\pi_1(M)$ is finitely generated.
\end{thm}

It is unclear to the author whether the converse of Theorem \ref{escape_onehalf} holds for nonnegative Ricci curvature.

Throughout the paper, we always assume that $M$ is an open (non-compact and complete) Riemannian manifold with $\mathrm{Ric}\ge 0$ and $E(M,p)\not=1/2$; in particular, $\pi_1(M)$ is finitely generated. When $\pi_1(M)$ is an infinite group, by \cite{Mil,Gro_poly} and Proposition \ref{torsion_free_index}, $\pi_1(M)$ has a finitely generated torsion-free nilpotent subgroup $\mathcal{N}$ of finite index.

We gather some elementary lemmas from \cite{Pan_cone} about the basics of escape rate.

\begin{lem}\cite[Lemma 1.5]{Pan_cone}\label{escape_index}
	Let $(M,p)$ be an open manifold with $\mathrm{Ric}\ge 0$ and let $F: (\hat{M},\hat{p})\to (M,p)$ be a finite cover. Then $E(\hat{M},\hat{p})\le E(M,p)$.
\end{lem}

In practice, we choose a torsion-free nilpotent subgroup $\mathcal{N}$ of $\pi_1(M,p)$ with finite index and set $\hat{M}=\widetilde{M}/\mathcal{N}$ as a finite cover of $M$. By Lemma \ref{escape_index}, we have 
$$E(\hat{M},\hat{p})\le E(M,p)<1/2,\quad \pi_1(\hat{M},\hat{p})=\mathcal{N}.$$
Therefore, without loss of generality, we can replace $(M,p)$ by $(\hat{M},\hat{p})$ and then assume that $\pi_1(M,p)=\mathcal{N}$ is a finitely generated torsion-free nilpotent group. 

\begin{lem}\cite[Lemma 2.1 and Proposition 2.2]{Pan_cone}\label{midpt_orb}
	Let $(Y,y,G)\in\Omega(\widetilde{M},\Gamma)$. For any point $gy\in Gy$ that is not $y$, there is a minimal geodesic $\sigma$ from $y$ to $gy$ and an orbit point $g'y\in Gy$ such that
	$$d(m,g'y)\le E\cdot d(y,gy),$$
	where $m$ is the midpoint of $\sigma$. As a consequence, the orbit $Gy$ is connected for all $(Y,y,G)\in \Omega(\widetilde{M},\Gamma)$.
\end{lem}

The inequality in \cite[Lemma 2.1]{Pan_cone} states $d(m,g'y)<(1/2)\cdot d(y,gy)$. However, inspecting its proof, it is clear that a stronger inequality as stated above holds.

Lastly, we mention that the notion of escape rate and its properties can be naturally extended to group actions on non-compact length metric spaces.

\begin{defn}\label{defn_es_rate_general}
	Let $(X,p)$ be a complete non-compact length metric spaces and let $G$ be a group that acts effectively and isometrically on $X$. For each $g\in G$, we denote $c_g$ as a minimal geodesic from $p$ to $g\cdot p$ and we put
	$$|g|=d(p,g\cdot p)=\mathrm{length}(c_g).$$
	We define $E(X,p,G)$, the \textit{escape rate} of $(X,p,G)$, by
	$$E(X,p,G):=\limsup_{|g|\to \infty} \dfrac{R(c_g)}{|g|},$$
	where $R(c_g)$ is the infimum of all $R>0$ such that $c_g$ is contained in the $R$-tubular neighborhood of $G\cdot p$. If the orbit $G\cdot p$ is bounded, then we set $E(X,p,G)=0$ as a convention.
\end{defn}

\section{Motivating examples}\label{sec_exmp}

%\subsection{Wei's examples}\label{subsec_exmp_wei}

To better motivate Theorem A, we have a brief review of Wei's construction \cite{Wei} of open manifolds with $\mathrm{Ric}>0$ and torsion-free nilpotent fundamental groups. We explain the choice of warping functions and its relation to Hausdorff dimension of asymptotic orbits. Also see \cite{PW_ex}, \cite[Appendix B]{Pan_es0}, and \cite[Appendix A]{Pan_cone}.

For simplicity, we use the discrete Heisenberg $3$-group 
$$\Gamma=\left\{ 
\begin{pmatrix}
	1 & a & c\\
	0 & 1 & b\\
	0 & 0 & 1
\end{pmatrix}	
\bigg| a,b,c\in\mathbb{Z}  \right\}\subseteq 
\left\{ 
\begin{pmatrix}
	1 & a & c\\
	0 & 1 & b\\
	0 & 0 & 1
\end{pmatrix}	
\bigg| a,b,c\in\mathbb{R}  \right\}=:\widetilde{N}.$$
$\Gamma$ is torsion-free nilpotent with nilpotency step $2$. Let $\mathfrak{n}$ be the Lie algebra of $\widetilde{N}$. $\mathfrak{n}$ has a basis
$$X_0=\begin{pmatrix}
	0 & 1 & 0\\
	0 & 0 & 0\\
	0 & 0 & 0
\end{pmatrix}, \quad X_1=\begin{pmatrix}
	0 & 0 & 0\\
	0 & 0 & 1\\
	0 & 0 & 0
\end{pmatrix}, \quad X_2=\begin{pmatrix}
	0 & 0 & 1\\
	0 & 0 & 0\\
	0 & 0 & 0
\end{pmatrix}$$ 
with $[X_0,X_1]=X_2$ as the only nontrivial Lie bracket relation.

Let
$$h_0(r)=h_1(r)=(1+r^2)^{-\alpha},\quad h_2(r)=(1+r^2)^{-\beta},$$
where $\alpha,\beta>0$ to be specified later. We define a family of inner products with parameter $r\in [0,\infty)$ on $\mathfrak{n}$ by
$$||X_j||_r:=h_j(r) $$
and setting $X_i$ and $X_j$ orthogonal when $i\not= j$. This naturally defines a family of left-invariant metrics $g_r$ on $\widetilde{N}$. Note
$$||[X_0,X_1]||=||X_2||=(1+r^2)^{-\beta}=(1+r^2)^{-\beta+2\alpha}||X_0||_r||X_1||_r.$$
It follows from basic curvature formulas for left-invariant metrics on Lie group that
$$|\mathrm{sec}(g_r)|\le C\cdot (1+r^2)^{-2\beta+4\alpha},$$
where $C$ is a constant. Therefore, for $g_r$ to be almost flat as $r\to\infty$, we should require
$\beta>2\alpha$.

Next, we define a doubly warped product 
$$\widetilde{M}=[0,\infty) \times_{f(r)} S^{k-1} \times (\widetilde{N},g_r), \quad g= dr^2 + f(r)^2 ds_{k-1}^2 + g_r,$$
where $f(r)=r(1+r^2)^{-1/4}$. $\widetilde{M}$ is diffeomorphic to $\mathbb{R}^k \times \widetilde{N}$. Following a similar computation as in \cite{Wei,BW}, for $g$ to have $\mathrm{Ric}>0$ when $k$ is sufficiently large, the term
$|\mathrm{Ric}(g_r)|$ should be comparable to or smaller than $(1+r^2)^{-1}$; equivalently, we require that
$$\beta\ge 2\alpha+1/2.$$
To construct $M$, noting that $\widetilde{N}$ naturally acts on $\widetilde{M}$ freely by isometries, we can take the quotient Riemannian manifold $M=\widetilde{M}/\Gamma$. Then it is clear that $M$ satisfies $\mathrm{Ric}>0$ and $\pi_1(M)=\Gamma$.

We check the Hausdorff dimension of asymptotic orbits for the above example. Let $\gamma_j=\exp(X_j)\in \Gamma$ and let $\tilde{p}=(0,\star,\mathrm{id})\in \widetilde{M}$. We use $|\cdot|$ to denote the displacement of an isometry at $\tilde{p}$. Following the length estimate in \cite[Section 1.2]{PW_ex}, as $b\to\infty$ we have 
$$|\gamma_0^b|=|\gamma_1^b|\sim b^{\frac{1}{1+2\alpha}},\quad |\gamma_2^{(b^2)}|=|[\gamma_0^b,\gamma_1^b]|\le C\cdot b^{\frac{1}{1+2\alpha}}.$$
$|\gamma_2^b|$ can also be estimated by $h_2(r)$. Since 
$$\dfrac{1}{1+2\beta}\le \dfrac{1}{2(1+2\alpha)},$$
we have 
$$|\gamma_2^b|\sim b^{\frac{1}{1+2\beta}}.$$
After blowing-down
$$(r_i^{-1}\widetilde{M},\tilde{p},\langle \gamma_j \rangle,\Gamma)\overset{GH}\longrightarrow (Y,y,H_j,G),$$
each $H_j$ is a closed $\mathbb{R}$-subgroup of $G$. It follows from the length estimate that the asymptotic orbits have Hausdorff dimension
$$\dimH(H_0y)=\dimH(H_1 y)=1+2\alpha,\quad \dimH(H_2 y)=1+2\beta >2.$$

\begin{rem}\label{rem_abel_limit}
As observed in \cite[Appendix A]{Pan_cone}, we also mention that the algebraic structure of $G$ depends on $\beta$. In fact, if $\beta=2\alpha+1/2$, then $G$ is isomorphic to the $3$-dimensional Heisenberg group $\widetilde{N}$; if $\beta>2\alpha+1/2$, then $G$ is isomorphic to the abelian group $\mathbb{R}^3$.
\end{rem}

\section{Equivariant Gromov-Hausdorff distance gaps}\label{sec_egh_gap}

The main goal of this section is to establish equivariant Gromov-Hausdorff distance gaps between different types of equivariant asymptotic cones of $(\widetilde{M},\mathcal{N})$ (Propositions \ref{eGH_gap_1} and \ref{eGH_gap_2}).

In subsection \ref{subsec_egh_tunnel}, we introduce the notion of tunnels, that is, continuous curves inside the orbit $Gy$. Using the condition $E(M,p)<1/2$, we study controls on the size of tunnels in the asymptotic cones and the convergence of tunnels. In subsection \ref{subsec_egh_idea}, we give the statements of distance gaps and a rough idea of the proof by using the convergence of tunnels and large fiber lemma. Subsection \ref{subsec_egh_adapt} is a technical part that we introduce some tools so that we can carry out the rough idea in general nilpotent group actions; in particular, we introduce the notion of adapted bases and adapted maps. We prove the distance gaps Propositions \ref{eGH_gap_1} and \ref{eGH_gap_2} in subsection \ref{subsec_egh_proof}.

\subsection{Tunnels}\label{subsec_egh_tunnel}

In this subsection, we use $E$ exclusively to denote the value of $E(M,p)<1/2$.

\begin{defn}\label{def_tunnel}
	Let $(Y,y,G)$ be a space. A \textit{tunnel} is a continuous path $\sigma:[0,1] \to Gy$. We say that $(Y,y,G)$ is \textit{$C$-tunneled} for some constant $C\in [1,\infty)$, if for every orbit point $gy\in Gy$ with $d(gy,y)=:d$, there is a tunnel $\sigma$ from $y$ to $gy$ that is contained in $\overline{B_{Cd}}(y)$.
\end{defn}

\begin{rem}
We remark that, in general, it is possible that every nontrivial curve in $Gy$ is has infinite length. For example, as constructed in \cite{PW_ex} and clarified in \cite[Remark 3.9]{DHPW}, the Grushin halfplane $\mathbb{G}(\alpha)$, or more generally the Grushin halfplane $\mathbb{G}(\alpha_1,...,\alpha_k)$ mentioned in the introduction, are asymptotic cones of open manifolds with $\mathrm{Ric}\ge 0$. The orbit $Gy$ is exactly $\mathbb{R}^k \times \{0\}$ under the  $\mathbb{R}^k \times [0,\infty)$-coordinate of $\mathbb{G}(\alpha_1,...,\alpha_k)$. When all $\alpha_i$ are positive, every nontrivial curve in $Gy$ is not rectifiable. Therefore, we use size, instead of length, to measure tunnels in Definition \ref{def_tunnel}.
\end{rem}

\begin{prop}\label{C_tunnel}
	Given $E\in[0,1/2)$, there is a constant $C_0(E)$ such that the following holds.
	
	Let $M$ be an open manifold with $\mathrm{Ric}\ge 0$ and $E(M,p)=E$. Then any $(Y,y,G)\in\Omega(\widetilde{M},\Gamma)$ is $C_0(E)$-tunneled.
\end{prop}

\begin{proof}
	We put 
	$$C_0(E)=\sum_{j=1}^\infty \left(E+1/2\right)^j<\infty.$$
	Let $gy\in Gy-\{y\}$ with $d:=d(y,gy)$. We use Lemma \ref{midpt_orb} to construct a desired tunnel $\sigma$ from $y$ to $gy$ as follows. We define $\sigma(0)=y$ and $\sigma(1)=gy$. By Lemma \ref{midpt_orb}, there is an orbit point $g_{(1,1)}y\in Gy$ such that
	$$d(g_{(1,1)}y,y)\le \left(E+1/2\right)d,\quad d(g_{(1,1)}y,gy)\le \left(E+1/2\right)d.$$
	We define $\sigma(1/2)=g_{(1,1)}y$. Next, we set
	$$P_j=\{k/2^j|k=0,1,...,2^j\}.$$
	Inductively, suppose that we have defined $\sigma$ on $P_j$ with $\sigma(k/2^j)=g_{(k,j)}y$ for some $g_{(k,j)}\in G$ such that
	$$d(g_{(k-1,j)}y,g_{(k,j)}y)\le \left(E+1/2\right)^{j}d.$$
	for all $k=1,...,2^j$. Then we use Lemma \ref{midpt_orb} again to define $\sigma$ on $P_{j+1}$ as follows. 	When $k$ is even, then $k/2^{j+1}\in P_j$; hence $\sigma(k/2^{j+1})$ has been defined in the previous steps. When $k$ is odd, we assign $\sigma(k/2^{j+1})$ as an orbit point $g_{(k,j+1)}y$ with
	$$d(g_{(k,j+1)}y,g_{(\frac{k-1}{2},j)}y)\le \left(E+1/2\right)d(g_{(\frac{k-1}{2},j)}y,g_{(\frac{k+1}{2},j)}y)\le \left(E+1/2\right)^{j+1}d,$$
	$$d(g_{(k,j+1)}y,g_{(\frac{k+1}{2},j)}y)\le \left(E+1/2\right)^{j+1}d.$$

	So far, we have defined $\sigma$ on $\cup_j P_j$ with image in $Gy$. We show that $\sigma$ is uniform continuous on $\cup_j P_j$. In fact, for any $\epsilon>0$, let $J\in\mathbb{N}$ large with 
	$$\sum_{j=J}^\infty (1/2+E)^j\le\epsilon/2,$$ 
	and we choose $\delta=1/2^{J+1}$. Then for any $t_1,t_2\in[0,1]$ with $|t_1-t_2|\le \delta$, there is some $k_0/2^{J}\in P_{J}$ such that 
	$$\max\{|k_0/2^{J}-t_1|,|k_0/2^{J}-t_2|\}|\le 1/2^J.$$
	Then by construction of $\sigma$,
	\begin{align*}
		d(\sigma(t_1),\sigma(t_2))\le& d(\sigma(t_1),\sigma(k_0/2^{J}))+d(\sigma(k_0/2^{J}),\sigma(t_2))\\
		\le & 2\sum_{j=J}^\infty (1/2+E)^j\le \epsilon.
	\end{align*}
	This verifies that $\sigma$ is uniform continuous on $\cup_j P_j$, thus $\sigma$ extends to a continuous path from $y$ to $gy$ in $Gy$.
	
	Lastly, observe that by construction, 
	$$d(y,\sigma(t))\le \sum_{j=1}^\infty (1/2+E)^{j}d = C_0(E)d$$
	for all $t\in[0,1]$.
	Therefore, the image of $\sigma$ is contained in $\overline{B_{C_0(E)d}}(y)$. This proves that $(Y,y,G)$ is $C_0(E)$-tunneled.
\end{proof}

\begin{rem}
	Regarding Proposition \ref{C_tunnel}, we mention that a stronger result holds: $(Y,y,G)$ is $3$-tunneled. Hence the size of the tunnel actually can be uniformly controlled regardless of the value of $E$, as long as $E\not=1/2$. Because Proposition \ref{C_tunnel} is sufficient for this paper, we omit the proof of this stronger result.
\end{rem}
%
%\begin{lem}\label{nearby_tunnel}
%	Let $(Y,y,G)$ and $(Z,z,H)$ be two spaces. Suppose that\\
%	(1) $d_{GH}((Y,y,G),(Z,z,H))\le\epsilon$,\\
%	(2) $(Y,y,G)$ is $C_0$-tunneled.\\
%	Then for any tunnel $\sigma:[0,1]\to Hz$ from $z$, there is a tunnel $\sigma':[0,1]\to Gy$ from $y$ such that 
%	$$d_{GH}(\sigma'(t),\sigma(t)) \le (3C_0+2)\epsilon$$
%	for all $t\in [0,1]$.
%\end{lem}

\begin{lem}\label{tunnel_conv}
	Let $(Y_i,y_i,G_i)$ be a convergent sequence of spaces
	$$(Y_i,y_i,G_i)\overset{GH}\longrightarrow (Z,z,H),$$
	where the limit group $H$ is nilpotent. Suppose that there is $C_0$ such that $(Y_i,y_i,G_i)$ is $C_0$-tunneled for all $i$.
	Then for any tunnel $\sigma:[0,1]\to Hz$ from $z$, there is a sequence of tunnels $\sigma_i:[0,1]\to G_iy_i$ from $y_i$ such that $\sigma_i$ converges uniformly to $\sigma$.
\end{lem}

\begin{proof}
	Let $$\epsilon_i=d_{GH}((Y_i,y_i,G_i),(Z,z,H))\to 0.$$
	We shall construct $\sigma_i$ on $(Y_i,y_i,G_i)$ for each large $i$. By the uniform continuity of $\sigma$, we can choose a large integer $N$ such that 
	$$\mathrm{diam}(\sigma|_{[(j-1)/N,j/N]})\le\epsilon_i$$
	for all $j=1,2,..,N$. Let $z_j=\sigma(j/\mathcal{N})\in Hz$. For each $j$, we choose $y_{i,j}\in G_iy_i$ that is $\epsilon_i$-close to $z_j$; for $j=0$, we use $y_{i,0}=y$. By triangle inequality, it is clear that
	$$d(y_{i,j},y_{i,j+1})\le 3\epsilon$$
	for all $j$. Next, for two adjacent $y_{i,j}$ and $y_{i,j+1}$, we join them by a tunnel 
	$$\sigma_{i,j}:[(j-1)/N,j/N]\to G_iy_i$$ 
	such that it is contained in $\overline{B_{3C_0\epsilon}}(y_{i,j})$. Let $\sigma_i:[0,1]\to Gy$ be the concatenation of all $\sigma_{i,j}$, where $j=1,...,N$. For any $t\in [0,1]$, let $j\in\{1,...,N\}$ such that $t\in[(j-1)/N,j/N]$, then by construction we have
	\begin{align*}
		d(\sigma_i(t),\sigma(t))&\le d(\sigma_i(t),y_{i,j})+d(y_{i,j},z_j)+d(z_j,\sigma(t))\\
		&\le  3C_0\epsilon_i + \epsilon_i + \epsilon_i =(3C_0+2)\epsilon_i\to 0.
	\end{align*}
\end{proof}

%\begin{lem}\label{tunnel_to_cnt_orb}
%	Let 
%	$$(Y_i,y_i,G_i)\overset{GH}\longrightarrow (Z,z,H)$$ 
%	be a convergent sequence of $C_0$-tunneled spaces. Then the limit orbit $Hz$ is connected.
%\end{lem}
%
%\begin{proof}
%	Let $gz\in Hz$ be any orbit point. We show that $gz$ belongs to the connected component of $Hz$ containing $z$. Let $g_iy_i\in G_iy_i$ be a sequence of orbit points that converges to $gz$. We write $d=d(gz,z)$ Because $G_iy_i$ is connected, we can assume that $g_i\in (G_i)_0$. Because each $(Y_i,y_i,G_i)$ is $C_0$-tunneled, there is a tunnel $\sigma:[0,1]\to G_iy_i$ from $y_i$ to $g_iy_i$ such that $\mathrm{im}\sigma_i \subseteq \overline{B_{2C_0 d}}(y_i)$. As a sequence of connected and closed subsets, $\mathrm{im}\sigma_i$ sub-converges to a limit connected and closed subset $S\subseteq \overline{B_{2C_0 d}}(z)\cap Hz$. By construction, $S$ contains both $z$ and $gz$. Thus $z$ and $gz$ belong to the same connected component of $Hz$. 
%\end{proof}

%We remark that in the above proof, the limit $S$ may not be path-connected.

Let $G$ be a Lie group. We use $G_0$ to denote the identity component subgroup of $G$. In a space $(Y,y,G)$, where $G$ is a Lie group, if $Gy$ is connected, then $Gy=G_0y$; consequently, every orbit point $z\in Gy$ can be represented as $gy$ for some $g\in G_0$.

\begin{lem}\label{tunnel_by_isotropy}
	Let 
	$$(Y_i,y_i,G_i)\overset{GH}\longrightarrow (Z,z,H)$$ 
	be a convergent sequence such that each $(Y_i,y_i,G_i)$ is $C_0$-tunneled and $H$ is nilpotent. Suppose that a sequence of orbit points $g_iy_i\in G_iy_i$ converges to a limit orbit point $gz\in Hz$, where $g_i\in (G_i)_0$ and $g\in H_0$. Then after passing to a subsequence if necessary, $g_i$ converges to $gh\in H_0$ with $h\in \mathrm{Iso}(z,H_0)$, the isotropy subgroup of $H_0$ at $z$.
\end{lem}

\begin{proof}
	Let $d=d(z,gz)$. Passing to a subsequence, we have convergence
	$$(Y_i,y_i,G_i,g_i)\overset{GH}\longrightarrow (Z,z,H,g_\infty),$$
	where $g_\infty \in H$. 
	
	We claim that $g_\infty \in H_0$. We argue by contradiction. Suppose that $g_\infty\not\in H_0$, then by Proposition \ref{G0_by_orbit} there is a point $q\in Z$ such that the orbit $Hq$ has multiple components and $g_\infty q$ is not in the component containing $q$. On $(Y_i,y_i,G_i)$, let $\sigma_i:[0,1]\to G_iy_i$ be a tunnel from $y_i$ to $g_iy_i$ such that $\mathrm{im}\sigma_i \subseteq \overline{B_{2C_0d}}(y_i)$. %By the nilpotency of $(G_i)_0$, the orbit $G_iy_i$ is homeomorphic to the base manifold of a fiber bundle $$T_i\to (G_i)_0 \to G_iy_i,$$
	%where $T_i=\mathrm{Iso}(y_i,(G_i)_0)$ is central in $(G_i)_0$. 
    Let $\widetilde{\sigma_i}:[0,1]\to (G_i)_0$ be a continuous path from $\mathrm{id}$ to $g_i$ such that $\sigma_i(t)=\widetilde{\sigma_i}(t)\cdot y_i$. Let $q_i\in Y_i$ converging to $q$. We consider the continuous path $\tau_i(t):=\widetilde{\sigma_i}(t)\cdot q_i$ in $G_iq_i$ from $q_i$ to $g_iq_i$. Then
	$$d(\tau_i(t),y_i)\le d(\widetilde{\sigma_i}(t) q_i,\widetilde{\sigma_i}(t) y_i)+d(\widetilde{\sigma_i}(t) y_i,y_i)\le d(q_i,y_i)+2C_0d.$$
	We write $l=d(q,y)$. After passing to a subsequence, connected and closed subsets $\mathrm{im}\tau_i$ converges to a limit connected and closed subset $S\subseteq \overline{B_{l+2C_0d}(y)}\cap Hq$. Since $S$ contain both $q$ and $g_\infty q$, we conclude that $q$ and $g_\infty q$ belong to the same connected component of $Hq$. This proves the claim.
	 
	We have $g_i\to g_\infty$ and $g_iy_i\to gz$, thus $g_\infty z=gz$. Let $h=g^{-1}g_\infty$, then $h$ fixes $z$. Together with $g\in H_0$ and the claim $g_\infty\in H_0$, we conclude that $h\in \mathrm{Iso}(z,H_0)$.
\end{proof}

\subsection{Statements of distance gaps and a rough idea of the proof}\label{subsec_egh_idea}

As explained in Lemma \ref{escape_index}, without loss generality, we assume that $\pi_1(M,p)=\mathcal{N}$ is a finitely generated torsion-free nilpotent group. Then for any $(Y,y,G)\in\Omega(\widetilde{M},\mathcal{N})$, $G$ is a nilpotent Lie group.

\begin{defn}\label{def_type_kd}
	Let $(Y,y,G)$ be a space, where $G$ is a nilpotent Lie group. Let $T$ be the maximal torus of $G_0$. We say that $(Y,y,G)$ is of type $(k,d)$, if 
	$$\dim G - \dim T =k,\quad \mathrm{diam}(Ty)=d.$$
\end{defn}

Using Definition \ref{def_type_kd}, we state the equivariant Gromov-Hausdorff distance gaps.

\begin{prop}\label{eGH_gap_1}
	There is a constant $\delta_1=\delta_1(\widetilde{M},\mathcal{N})>0$ such that the following holds. 
	
	Let $(Y,y,G)$ and $(Y',y',G')\in\Omega(\widetilde{M},\mathcal{N})$ of type $(k,d)$ and $(k',d')$, respectively. Suppose that $k<k'$ and $d\le 1$, then
	$$d_{GH}((Y,y,G),(Y',y',G'))\ge \delta_1.$$ 
\end{prop}

\begin{prop}\label{eGH_gap_2}
	There is a constant $\delta_2=\delta_2(\widetilde{M},\mathcal{N})>0$ such that the following holds. 
	
	Let $(Y,y,G)$ and $(Y',y',G')\in\Omega(\widetilde{M},\mathcal{N})$ of type $(k,d)$ and $(k',d')$, respectively. Suppose that $k=k'$, $d\le 1$, and $d'\ge 10$, then
	$$d_{GH}((Y,y,G),(Y',y',G'))\ge \delta_2.$$ 
\end{prop}

Besides using tunnels, a key ingredient in the proof is the large fiber lemma from topological dimension theory.

\begin{lem}[Large Fiber Lemma]\label{large_fiber}
	Let $F:[0,1]^{k+1}\to \mathbb{R}^k$ be a continuous map. Then there are $a,b\in [0,1]^{k+1}$ such that $F(a)=F(b)$ and $|a-b|\ge 1$. 
\end{lem}

\begin{rem}
The large fiber lemma is a corollary of the Lebesgue covering lemma in topological dimension theory (see \cite[Section 6]{Guth}). It also follows from the Borsuk-Ulam theorem in algebraic topology (see \cite[Corollary 2B.7]{Hatcher}): we take a $k$-sphere of radius $1/2$ in $[0,1]^{k+1}$, then by Borsuk-Ulam theorem, there exists a pair of antipodal points on the sphere such that they have the same image under $F$.
\end{rem}

To illustrate how the large fiber lemma and the $C$-tunneled property can be applied to prove equivariant Gromov-Hausdorff distance gaps between two spaces, we rule out the following scenario.
Suppose that there is a sequence
$$(Y_i,y_i,G_i)\overset{GH}\longrightarrow (Z,z,H)$$
such that\\
(1) for each $i$, $(Y_i,y_i,G_i)$ is $C_0$-tunneled and $G_i$ is isomorphic $\mathbb{R}^k$;\\
(2) $H$ is isomorphic to $\mathbb{R}^{k+1}$.

Let $\{e_1,...,e_{k+1}\}\subseteq H$ be an $\mathbb{R}$-basis of $H=\mathbb{R}^{k+1}$. Note that $H$-action does not have isotropy subgroups at $z$ (otherwise, $H$ would have a nontrivial compact subgroup). For each $j=1,...,k+1$, we write $\{te_j\}_{t\in\mathbb{R}}$ as the one-parameter subgroup through $e_j$. We consider an embedding 
$$F:[0,1]^{k+1} \to Hz, \quad (t_1,...,t_{k+1})\mapsto \prod_{j=1}^{k+1} (t_je_j)\cdot z.$$
$\sigma_j(t):=(te_j)z$, where $t\in[0,1]$, gives a tunnel in $Hz$ from $z$ to $e_jz$. We apply Lemma \ref{tunnel_conv} to construct tunnels $\sigma_{i,j}:[0,1]\to G_iy_i$ from $y_i$ that converges uniformly to $\sigma_j$ as $i\to\infty$. Because $G_i$-action does not have isotropy subgroups at $y_i$, each $\sigma_{i,j}$ uniquely defines a continuous curve $\widetilde{\sigma_{i,j}}:[0,1]\to G_i$ such that $\widetilde{\sigma_{i,j}}(t)\cdot y_i=\sigma_{i,j}(t)$. This allows us to define a continuous map
$$F_i:[0,1]^{k+1} \to G_iy_i,\quad (t_1,...,t_{k+1})\mapsto \prod_{j=1}^{k+1} \widetilde{\sigma_{i,j}}(t_j) \cdot y_i.$$
By construction, it is not difficult to show that $F_i$ converges uniformly to $F$. Since $G_iy_i$ is homeomorphic to $\mathbb{R}^k$, we can apply the large fiber lemma to $F_i$. It follows that there are $a_i,b_i\in [0,1]^{k+1}$ such that $F_i(a_i)=F_i(b_i)$ and $|a_i-b_i|\ge 1$. Passing to a subsequence, we have limit points $a',b'\in [0,1]^{k+1}$ such that $F(a')=F(b')$ and $|a'-b'|\ge 1$ by the uniform convergence of $F_i$ to $F$. However, this contradicts the injectivity of $F$.

\begin{rem}
	We remark that in general, it is possible for a sequence of $\mathbb{R}^k$-actions converges to a limit $\mathbb{R}^{k+1}$-action, as shown in the example below. Hence the $C$-tunneled property is crucial here. 
	
	We consider a sequence of $\mathbb{R}$-actions on the standard Euclidean space $\mathbb{R}^3$ as follows. Set $p=(0,0,0)$ as our base point. For each $i\in\mathbb{Z}_+$, $G_i=\mathbb{R}$ acts on $\mathbb{R}^3$ by rotating $xy$-plane by angle $2\pi t$ with center $(i,0,0)$ while translating along $z$-axis by $t/i$, where $t\in\mathbb{R}$. Then one can directly check that
	$$(\mathbb{R}^3,p,G_i)\overset{GH}\longrightarrow (\mathbb{R}^3,p,\mathbb{R}^2).$$
	The limit $\mathbb{R}^2$-action are translations in $yz$-plane. Note that these spaces $(\mathbb{R}^3,p,G_i)$ are not $C$-tunneled for a uniform $C$.
\end{rem}

\subsection{Adapted bases and adapted maps}\label{subsec_egh_adapt}

In general, the group actions involved are nilpotent and may have nontrivial torus subgroups that move the base point. Hence we need more preparations to carry out the strategy in subsection \ref{subsec_egh_idea}.

For convenience, we define
\begin{align*}
	\Omega_Q(\widetilde{M},\mathcal{N})=\{(Y/H,\bar{y},G/H)|& (Y,y,G)\in\Omega(\widetilde{M},\mathcal{N}),\\
	& H \text{ is a closed normal subgroup of }G \}.
\end{align*}
Note that $\Omega_Q(\widetilde{M},\mathcal{N})$ includes $\Omega(\widetilde{M},\mathcal{N})$ since we can take $H=\{\mathrm{id}\}$.

\begin{defn}\label{def_good}
	We say that a space $(Y,y,G)$ is \textit{good}, if the followings hold:\\
	(1) $(Y,y)\in \Omega_Q(\widetilde{M},\mathcal{N})$;\\
	(2) $G\subseteq \mathrm{Isom}(Y)$ is closed nilpotent subgroup;\\
	(3) $(Y,y,G)$ is $C_0$-tunneled, where $C_0$ is the constant in Proposition \ref{C_tunnel};\\
	(4) The isotropy subgroup of $G$ at $y$ is finite;\\
	(5) $G_0$, the identity component subgroup of $G$, is simply connected.
\end{defn}

\begin{lem}\label{lem_quo_tori_good}
	Let $(Y,y,G)\in \Omega(\widetilde{M},\mathcal{N})$ and let $T$ be the maximal torus subgroup of $G_0$. Then the quotient space $(Y/T,\bar{y},G/T)$ is good in the sense of Definition \ref{def_good}.
\end{lem}

\begin{proof}
	We first remark that by Lemma \ref{max_torus} $T$ is normal in $G$, thus the quotient group $G/T$ is defined. It is clear that (1,2) in Definition \ref{def_good} are fulfilled. (4,5) are also straightforward since $(G/T)_0=G_0/T$ does not have any nontrivial torus subgroup.
	
	It remains to show (3). Let $\pi:Y\to Y/T$ be the quotient map. Since $\pi$ maps $Gy$ to $(G/T)\bar{y}$ and $Gy$ is connected by Lemma \ref{midpt_orb}, we see that $(G/T)\bar{y}$ is also connected. For any orbit point $\bar{z}\in (G/T)\bar{y}$, because $(G/T)\bar{y}$ is connected, we can write $\bar{z}=\bar{g}\bar{y}$, where $\bar{g}\in (G/T)_0=G_0/T$. We choose $g\in G_0$ such that $g$ projects to $\bar{g}$ and 
	$$d_Y(gy,y)=d_{Y/T}(\bar{g}\bar{y},\bar{y})=:d.$$
	By Lemma \ref{C_tunnel}, there is a tunnel $\sigma:[0,1]\to Gy$ from $y$ to $gy$ that is contained in $\overline{B_{C_0d}}(y)$. Then $\pi\circ\sigma$ is a desired tunnel from $\bar{y}$ to $\bar{g}\bar{y}$ in $(G/T)\bar{y}$.
\end{proof}

We first construct adapted bases and maps for good spaces in the sense of Definition \ref{def_good}. In this case, the maximal torus of $G_0$ is trivial.

\begin{defn}\label{def_initial}
	Let $(Y,y,G)$ be a space, where $G_0$ is a simply connected nilpotent Lie group. Let $$G_0=\zeta_0(G_0)\triangleright \zeta_1(G_0) \triangleright ...\triangleright \zeta_{l-1}(G_0)\triangleright \zeta_l(G_0)=\{\mathrm{id}\}$$
	be the lower central series of $G_0$, where $\zeta_{l-1}(G_0)\not=\{\mathrm{id}\}$.  We say an element $e_1\in G_0$ is \textit{initial}, if $e_1\in \zeta_{l-1}(G_0)$ and $d(e_1y,y)=1$.
	%(3) $d((te_1)y,y)<1$ for all $t\in (0,1)$.
\end{defn}

Note that every one-parameter subgroup of $\zeta_{l-1}(G_0)$ has an unbounded orbit at $y$. Thus the initial element defined above always exists.

This initial element $e_1$ is the first element of an adapted basis $\{e_1,...,e_k\}$ with respect to $(Y,y,G)$, where $k$ is the dimension of $G$. We choose the remaining elements by induction. Let $H_1$ be the unique one-parameter subgroup through $e_1$ (see Lemma \ref{nil_exp}). By construction, $H_1$ is normal in $G_0$. We choose $\bar{e}_2\in G_0/H_1$ as the initial element in $(Y/H_1,\bar{y},G_0/H_1)$. Let $e_2\in G_0$ such that $e_2$ projects to $\bar{e}_2\in G_0/H_1$ and
$$d_Y(e_2 y,y)=d_{Y/H_1}(\bar{e}_2\bar{y},\bar{y})=1.$$
Note that because $\bar{e}_2$ belongs to the last nontrivial subgroup of the lower central series of $G_0/H_1$, we have
$$ [v,v_2]\in \mathrm{span}_\mathbb{R} \{v_1\}$$
for all $v\in \mathrm{Lie}(G_0)$, where $v_j\in \mathrm{Lie}(G_0)$ such that $\exp(v_j)=e_j$. Inductively, we choose $\{e_1,...,e_k\}$ such that\\
(1) $[v,v_{j+1}]\in \mathrm{span}_\mathbb{R} \{v_1,...,v_{j}\}$ for all $j=1,2,...,k-1$, where $v_j\in \mathrm{Lie}(G_0)$ such that $\exp(v_j)=e_j$;\\
(2) $\bar{e}_{j+1}$ is an initial element of $(Y/H_j,\bar{y},G_0/H_j)$, where $H_j$ is the Lie subgroup with Lie algebra as $\mathrm{span}\{v_1,...,v_j\}$, and $e_{j+1}\in G_0$ such that $e_{j+1}$ projects to $\bar{e}_{j+1}$ and
$$d_Y(e_{j+1} y,y)=d_{Y/H_j}(\bar{e}_{j+1}\bar{y},\bar{y})=1.$$

\begin{defn}\label{def_adapted_basis_good}
	Let $(Y,y,G)$ be a good space in the sense of Definition \ref{def_good}. We call the above constructed $\{e_1,...,e_k\}\subseteq G_0$ an \textit{adapted basis} with respect to $(Y,y,G)$, where $k$ is the dimension of $G$. 
\end{defn}

As a convention, $\prod$ means a product multiplying on the left
$\prod_{j=1}^k g_j=g_k...g_2g_1.$

\begin{defn}\label{def_adapted_map_good}
	Let $(Y,y,G)$ be a good space in the sense of Definition \ref{def_good}. Let $\mathcal{E}=\{e_1,...,e_k\}\subseteq G_0$ be an adapted basis with respect to $(Y,y,G)$, where $k$ is the dimension of $G$. We define an \textit {adapted map} for $\mathcal{E}$:
	$$F:[0,1]^k \to Gy,\quad (t_1,...,t_k)\mapsto \prod_{j=1}^k (t_je_j) \cdot y,$$
	where $te_j$ denotes the elements on the unique one-parameter subgroup through $e_j$.
\end{defn}

\begin{lem}\label{adapted_map_inj_good}
	Let $(Y,y,G)$ be a good space in the sense of Definition \ref{def_good}. Let $F:[0,1]^k \to Gy$ be an adapted map for an adapted basis $\mathcal{E}$ as in Definition \ref{def_adapted_map_good}. Then $F$ is a continuous injection.
\end{lem}

\begin{proof}
	Note that $G_0$-action is free at $y$; otherwise $G_0$ would have a compact torus subgroup as the isotropy subgroup at $y$. Because $G_0$-action is continuous and free at $y$, it suffices to show that the map
	$$\widetilde{F}:[0,1]^k \to G_0,\quad (t_1,...,t_k)\mapsto \prod_{j=1}^k (t_je_j)$$
	is a continuous injection. It is clear that $\widetilde{F}$ is continuous. We prove its injectivity by induction on the nilpotency step of $G_0$. 
	
	When $G_0$ is abelian, it is clear that $\widetilde{F}$ is injective. Assuming that $\widetilde{F}$ is injective when $G_0$ has nilpotency step $\le l$, we consider the case that $G_0$ has nilpotency step $l+1$. For each $j$, let $v_j\in \mathrm{Lie}(G_0)$ such that $\exp(v_j)=e_j$. Suppose that
	$$\prod_{j=1}^k \exp(t_jv_j) = \prod_{j=1}^k \exp(s_jv_j).$$
	By the construction of the adapted basis, there is an integer $m\in [1,k)$ such that $\zeta_{l-1}(G_0)=\exp(V_m)$, where $V_m$ is the span of $\{v_1,...,v_m\}$. The quotient group 
	$G_0/\zeta_{l-1}(G_0)$ has nilpotency step $l$. Let $\bar{v_j}$, where $j>m$, be the projection of $v_j$ in the quotient Lie algebra $\mathrm{Lie}(G_0)/V_m=\mathrm{Lie}(G_0/\zeta_{l-1}(G_0))$. Then in $G_0/\zeta_{l-1}(G_0)$ we have
	$$\prod_{j=m+1}^k \exp(t_j\bar{v_j})=\prod_{j=m+1}^k \exp(s_j\bar{v_j}).$$
	By the induction assumption, we have $t_j=s_j$ for all $j>m$.
	By Lemma \ref{nil_exp}, there is an element $Z\in \mathrm{Lie}(G_0)$ such that
	$$\exp(Z)=\prod_{j=m+1}^k \exp(t_j {v_j}).$$
	Because $v_1,..,v_m$ are in the center of $\mathrm{Lie}(G_0)$, by the Baker–Campbell–Hausdorff formula, it follows that
	$$\exp\left(\sum_{j=1}^m t_jv_j+Z\right)=\exp\left(\sum_{j=1}^m s_jv_j+Z\right).$$
	By Lemma \ref{nil_exp} again, we see that
	$$\sum_{j=1}^m t_jv_j+Z=\sum_{j=1}^m s_jv_j+Z.$$
	We conclude that $t_j=s_j$ also holds for $j=1,...,m$. This completes the inductive step and thus $\widetilde{F}$ is injective.
\end{proof}

In general, we will also consider spaces that do not satisfy Definition \ref{def_good}. There are mainly two cases, corresponding to Propositions \ref{eGH_gap_1} and \ref{eGH_gap_2} respectively. We shall similarly construct adapted bases and adapted maps in each case.

For Proposition \ref{eGH_gap_1}, we consider a space $(Z,z,H)\in\Omega_Q(\widetilde{M},\mathcal{N})$ of type $(k,d)$. Let $T_H$ be the maximal torus subgroup of $H_0$. The quotient space $(Z/T_H,\bar{z},H/T_H)$ is a good space in the sense of Definition \ref{def_good} by Lemma \ref{lem_quo_tori_good}. Thus we can follow Definitions \ref{def_adapted_basis_good} and \ref{def_adapted_map_good} to construct an adapted basis $\{\bar{e}_1,...,\bar{e}_k\}\subseteq (H/T_H)_0=H_0/T_H$. %and an adapted map $$\bar{F}:[0,1]^k \to (G/T_H)\cdot \bar{z}=(H/T_H)_0\cdot \bar{z},\quad (t_1,...,t_k)\mapsto \prod_{j=1}^k (t_j\bar{e}_j) \cdot \bar{z}.$$
For each $j=1,...,k$, we choose $e_j\in H_0$ such that $e_j$ projects to $\bar{e_j}\in H/T_H$ and
$$d_Y(e_jz,z)=d_{Y/T_H} (\bar{e_j}\bar{z},\bar{z})=1.$$

\begin{defn}\label{def_adapted_ver1}
	Let $(Z,z,H)\in\Omega_Q(\widetilde{M},\mathcal{N})$ be a space of type $(k,d)$. We call the above constructed $\mathcal{E}=\{e_1,...e_k\}\subseteq H_0$ an \textit{adapted basis} with respect to $(Z,z,H)$. To further construct an adapted map, for each $e_k$, we choose a one-parameter subgroup $\tau_j:\mathbb{R}\to H_0$ such that $\tau_j(1)=e_j$; note that the choice of $\tau_j$ may not be unique. We define an \textit{adapted map} for $\mathcal{E}$ as follows:
    $$F:[0,1]^k\to Hz,\quad (t_1,...,t_k)\mapsto \prod_{j=1}^k \tau_j(t_j)\cdot z.$$
\end{defn}

\begin{lem}\label{adapted_map_inj_ver1}
	Let $(Z,z,H)\in\Omega_Q(\widetilde{M},\mathcal{N})$ be a space of type $(k,d)$ and let $T_H$ be the maximal torus subgroup of $H$. Let $\mathcal{E}=\{e_1,...e_k\}\subseteq H_0$ be an adapted basis with respect to $(Z,z,H,T_H)$ and let $F:[0,1]^k\to Hz$ be an adapted map for $\mathcal{E}$. Then $F$ is a continuous injection.
\end{lem}

\begin{proof}
	The continuity of $F$ is clear. We prove its injectivity. Recall that $\mathcal{E}=\{e_1,...,e_k\}$ is the lift of an adapted basis $\overline{\mathcal{E}}=\{\bar{e_1},...,\bar{e_k}\}$ with respect to the quotient space $(Z/T_H,\bar{z},H/T_H)$. Note that a one-parameter subgroup $\tau_j$ through $e_j$ projects to the unique one-parameter subgroup through $\bar{e_j}$. Let $\pi:Z\to Z/T_H$ be the quotient map.  By construction,
	$$(\pi\circ F)(t_1,...,t_k)=\prod_{j=1}^k \pi\circ\tau_j(t_j)\cdot \bar{z}=\prod_{j=1}^k t_j \bar{e}_j \cdot\bar{z}.$$
	Thus $\pi\circ F$ is the adapted map for $\overline{\mathcal{E}}$. By Lemma \ref{adapted_map_inj_good}, $\pi\circ F$ is injective, thus $F$ is injective as well.
\end{proof}

Next, we consider the scenario for Proposition \ref{eGH_gap_2}. Let $(Z,z,H)\in\Omega(\widetilde{M},\mathcal{N})$ be a space of type $(k,d)$ with $d\ge 5$.
Let $T_H$ be the maximal torus subgroup of $H_0$ and let $\{e_1,...,e_k\}$ be an adapted basis with respect to $(Z,z,H)$ as constructed in Definition \ref{def_adapted_ver1}. For the sake of a dimensional argument, we need an additional element from $T_H$. We choose an element $e_0\in T_H$ such that\\
(1) $d({e_0}{z},{z})=1$,\\
(2) there is a piece of one-parameter subgroup $\tau_0$ from $\mathrm{id}$ to $e_0$ such that $\tau_0|_{(0,1]}$ is outside $\mathrm{Iso}(z,H_0)$.

\begin{defn}\label{def_adapted_ver2}
	Let $(Z,z,H)\in\Omega_Q(\widetilde{M},\mathcal{N})$ be a space of type $(k,d)$, where $d\ge 5$, and let $T_H$ be the normal torus subgroup of $H$. We call the above constructed $\{e_0,e_1...,e_k\}\subseteq H_0$ an \textit{adapted basis} with respect to $(Y,y,H,T_H)$. Next, we construct an adapted map. For $e_0$, we have already chosen a piece of one-parameter subgroup $\tau_0$ from $\mathrm{id}$ to $e_0$. For each $e_j$, where $j=1,...,k$, let $\tau_j$ be a piece of a one-parameter subgroup from $\mathrm{id}$ to $e_j$. We define an \textit{adapted map} for $\mathcal{E}$:
    $$F:[0,1]^{k+1}\to Hz,\quad (t_0,t_1,...,t_k)\mapsto \prod_{j=0}^k \tau_j(t_j)\cdot z.$$
\end{defn}

\begin{lem}\label{adapted_map_inj_ver2}
	Let $(Z,z,H)\in\Omega_Q(\widetilde{M},\mathcal{N})$ be a space of type $(k,d)$, where $d\ge 5$. Let $\mathcal{E}=\{e_0,e_1,...e_k\}\subseteq H_0$ be an adapted basis with respect to $(Z,z,H,T_H)$ and let $F:[0,1]^{k+1}\to Hz$ be an adapted map for $\mathcal{E}$. Then $F$ is a continuous injection.
\end{lem}

\begin{proof}
	It is clear that $F$ is continuous. Suppose that 
	$$F(t_0,t_1,...,t_k)=F(t'_0,t'_1,...,t'_k).$$
	Let $\pi: Z\to Z/T_H$ be the quotient map. By the proof of Lemma \ref{adapted_map_inj_ver1}, $\pi\circ F$ is an adapted map and thus is injective. This shows that $t'_j=t_j$ for $j=1,...,k$. Now we have
	$$\left(\prod_{j=1}^k \tau_j(t_j)\right)\tau_0(t_0)z=\left(\prod_{j=1}^k  \tau_j(t_j)\right)\tau_0(t'_0)z.$$
	Thus $\tau_0(t_0)z=\tau_0(t'_0)z$. Because $\tau_0$ is constructed from a one-parameter subgroup, we have $\tau_0(t_0-t'_0)z=z$. Recall that $\tau|_{(0,1]}$ is outside $\mathrm{Iso}(z,H_0)$. Hence we must have $t_0=t'_0$. 
\end{proof}

To complete this subsection, we use the properties of tunnels (Lemmas \ref{tunnel_conv} and \ref{tunnel_by_isotropy}) to construct maps converging uniformly to an adapted map.

\begin{lem}\label{conv_to_adapted_map}
	Let 
	$$(Y_i,y_i,G_i)\overset{GH}\longrightarrow(Z,z,H)$$
	be a convergent sequence of spaces in $\Omega_Q(\widetilde{M},\mathcal{N})$. Suppose that\\
	(1) each $(Y_i,y_i,G_i)$ is good in the sense of Definition \ref{def_good};\\
	(2) on the limit space $(Z,z,H)$, there is an adapted map $F$ defined in either Definition \ref{def_adapted_ver1} or \ref{def_adapted_ver2} with domain $[0,1]^k$ or $[0,1]^{k+1}$, respectively.\\
	Then there is a sequence of continuous maps 
	$$F_i:[0,1]^k \text{ or } [0,1]^{k+1} \to G_iy_i\subseteq Y_i$$
	that converges uniformly to $F$. 
\end{lem}

\begin{proof}
	Before starting the proof, we remark that the limit space $(Z,z,H)$ may not be good in the sense of Definition \ref{def_good}. 
	
	Let $J=\{1,...,k\}$ or $\{0,1,...,k\}$. We use $\vv{t}$ to denote 
	$$\vv{t}=(t_1,...,t_k)\in [0,1]^k\ \text{ or }\ \vv{t}=(t_0,t_1,...,t_k)\in [0,1]^{k+1}.$$
	In both Definitions \ref{def_adapted_ver1} and \ref{def_adapted_ver2}, the adapted map has the form
	$$F:[0,1]^k \text{ or } [0,1]^{k+1} \to Hz,\quad  \vv{t} \mapsto \prod_{j\in J} \tau_j(t_j)\cdot z,$$
	where $\tau_j:[0,1]\to H_0$ is a piece of one-parameter subgroup from $\mathrm{id}$ to $e_j$, an element in the adapted basis $\mathcal{E}$. For each $j\in J$, let $\sigma_j(t)=\tau_j(t)\cdot z$, which is a tunnel from $z$ to $e_jz$. By Lemma \ref{tunnel_conv}, there is a sequence of tunnels $$\sigma_{i,j}:[0,1]\to G_iy_i$$
	from $y_i$ that converges uniformly to $\sigma_j$ as $i\to\infty$. Because each $(Y_i,y_i,G_i)$ is good, $(G_i)_0$ acts freely at $y_i$. Hence $\sigma_{i,j}$ uniquely determines a continuous path
	$$\widetilde{\sigma_{i,j}}:[0,1]\to (G_i)_0$$
	such that $\widetilde{\sigma_{i,j}}(t) \cdot y_i=\sigma_{i,j}(t)$. We construct $F_i$ as
	$$F_i:[0,1]^k \text{ or } [0,1]^{k+1} \to G_iy_i,\quad \vv{t} \mapsto \prod_{j\in J} \widetilde{\sigma_{i,j}}(t_j) \cdot y_i.$$
	
	We prove that $F_i$ converges uniformly to $F$. It suffices to show that for every convergent sequence
	$$(\vv{t})_i=(t_{i,j})_{j\in J}\to \vv{t}=(t_j)_{j\in J},$$
	it holds that $F_i((\vv{t})_i)\to F(\vv{t})$ as $i\to\infty$, that is,
	$$\prod_{j\in J} \widetilde{\sigma_{i,j}}(t_{i,j})\cdot y_i \to \prod_{j\in J} \tau_j(t_j)\cdot z$$
	given $t_{i,j}\to t_j$ for each $j\in J$. By construction of $\widetilde{\sigma_{i,j}}$, we have $\widetilde{\sigma_{i,j}}(t_{i,j}) y_i$ converging to $\tau_j(t_j)z$. After passing to a subsequence, we assume that for each $j\in J$, $\widetilde{\sigma_{i,j}}(t_{i,j})$ converges to some element in $H$ as $i\to\infty$. By Lemma \ref{tunnel_by_isotropy}, we have
	$$\widetilde{\sigma_{i,j}}(t_{i,j})\overset{GH}\to \tau_j(t_j)h_j, $$
	where $h_j\in \mathrm{Iso}(z,H_0)$. The compact subgroup $\mathrm{Iso}(z,H_0)$ must be contained in the maximal torus subgroup of $H_0$, thus each $h_j$ is central in $H_0$ by Lemma \ref{max_torus}. It follows that
	$$\prod_{j\in J} \widetilde{\sigma_{i,j}}(t_{i,j})\cdot y_i \to \prod_{j\in J} (\tau_j(t_j)h_j) \cdot z =\prod_{j\in J} \tau_j(t_j)  \cdot \prod_{j\in J} h_j  \cdot z= \prod_{j\in J} \tau_j(t_j)\cdot z.$$
	This verifies the uniform convergence of $F_i$ to $F$.
\end{proof}

\subsection{Proof of the distance gaps}\label{subsec_egh_proof}

We prove Propositions \ref{eGH_gap_1} and \ref{eGH_gap_2} in this subsection.

\begin{lem}\label{mod_max_torus}
	Let $(Y_i,y_i,G_i)$ be a sequence of spaces in $\Omega(\widetilde{M},\mathcal{N})$ and let $T_i$ denote the maximal torus subgroup of $G_i$. Suppose that there is $D>0$ such that $\mathrm{diam}(T_iy_i)\le D$ for all $i$ and
	$$(Y_i,y_i,G_i,T_i)\overset{GH}\longrightarrow (Z,z,H,K).$$ 
	Then
	$$(Y_i/T_i,\bar{y}_i,G_i/T_i)\overset{GH}\longrightarrow (Z/K,\bar{z},H/K).$$
\end{lem}

\begin{proof}
	The proof is standard by approximation maps. We give some details below for readers' convenience.
	
	Let $$\epsilon_i=10\cdot d_{GH}((Y_i,y_i,G_i),(Z,z,H))\to 0.$$
	It is clear that $\mathrm{diam}(Kz)\le D$. When $1/\epsilon_i \gg D$, we have a tuple of $\epsilon_i$-approximation maps $(f_i,\psi_i,\phi_i)$, that is,
	$$f_i:B_{1/\epsilon_i}(y_i)\to Z,\quad  \psi_i: G_i(1/\epsilon_i) \to H(1/\epsilon_i), \quad \phi_i:H(1/\epsilon_i) \to  G_i(1/\epsilon_i) $$
	with the properties (1)-(5) in Definition \ref{defn_FY} and\\
	(6) $\psi_i(T_i)\subseteq K$, $\phi_i(K)\subseteq T_i$. \\
	By Theorem \ref{eqGH_FY}(2), we have
	$$(Y_i/T_i,\bar{y_i})\overset{GH}\longrightarrow (Z/K,\bar{z}).$$
	Moreover, the approximation map $\bar{f}_i$ from $B_{1/(5\epsilon_i)}(\bar{y}_i)\subseteq Y_i/T_i$ to $Z/K$ can be chosen as an quotient of $f_i$; more precisely, for each $\bar{x}\in B_{1/(5\epsilon_i)}(\bar{y}_i)$, we define
	$$\bar{f}_i(\bar{x}):=\overline{f_i(x)}\in Z/K,$$
	where $x\in B_{1/(5\epsilon_i)}(y)$ is a point projecting to $\bar{x}\in Y_i/T_i$.
	
	%Since $\mathrm{diam}(T_iy_i)\le D \ll 1/\epsilon_i$, we have $T_i\subseteq G_i(1/\epsilon_i)$. 
	Let $\bar{g}\in \frac{G_i}{T_i}(\frac{1}{5\epsilon_i})$, then there are $t_1,t_2\in T_i$ and $g\in G_i$ projecting to $\bar{g}$ such that
	$$d(t_1g y_i,t_2 y_i)=d(Tgy,Ty)=d(\bar{g}\bar{y}_i,\bar{y}_i)\le \frac{1}{5\epsilon_i}.$$
	Thus 
	$$d(gy_i,y_i)\le d(gy_i,t_1 g y_i)+d(t_1 g y_i,t_2y_i)+d(t_2 y_i,y_i) \le D+\frac{1}{5\epsilon_i}+D.$$
	%This shows that
	%$$\frac{G_i}{T_i}\left(\frac{1}{5\epsilon_i}\right) \subseteq \dfrac{G_i(2D+1/(5\epsilon_i))}{T_i}.$$
	We define 
	$$ \bar{\psi}_i: \dfrac{G_i}{T_i}\left(\frac{1}{5\epsilon_i}\right) \to \dfrac{H}{K},\quad \bar{g} \mapsto \overline{\psi_i(g)}.$$
	We estimate
	\begin{align*}
		d(\overline{\psi_i(g)}\bar{z},\bar{z})&=d(K\psi_i(g)z,Kz)\\
		&= d(k_1\psi_i(g)z,k_2z) \ \ \text{for some } k_1,k_2\in K \\
		&\le d(\psi_i(g)z,z)+d(k_1z,z)+d(k_2z,z)\\
		&\le d(\psi_i(g)z,gy_i)+d(gy_i,y_i)+d(y_i,z)+2D\\
		&\le \epsilon_i+ \left[2D + 1/(5\epsilon_i)\right] +\epsilon_i + 2D\\
		&\le 1/(10\epsilon_i). 
	\end{align*}
	Thus $\mathrm{im}(\bar{\psi}_i)\subseteq \frac{H}{K}(\frac{1}{10\epsilon_i})$. For any $g\in \frac{G_i}{T_i}(\frac{1}{5\epsilon_i})$ and $\bar{x},\bar{g}\bar{x}\in B_{1/\epsilon_i}(\bar{y})\subset Y/T$, 
	\begin{align*}
		d(\bar{f}_i(\bar{g}\bar{x}),\bar{\psi}_i\bar{f}_i(\bar{x}))& = d(\overline{f_i(gx)},\overline{\psi_i(g)}\cdot \overline{f_i(x)}) \\
		& =d(K\cdot f_i(gx),K\cdot \psi_i(g)\cdot f_i(x))\\
		& \le d(f_i(gx),\psi_i(g)\cdot f_i(x))\\
		& \le \epsilon_i.
	\end{align*}
	Similarly, we can construct 
	$$\bar{\phi}_i: \frac{H}{K}\left(\dfrac{1}{5\epsilon_i}\right) \to \dfrac{G_i}{T_i},\quad \bar{h} \mapsto \overline{\phi_i(h)}$$
	with the desired estimates. Therefore, $(\bar{f}_i,\bar{\psi}_i,\bar{\phi}_i)$ gives $(10\epsilon_i)$-approximation maps between $(Y_i/T_i,\bar{y}_i,G_i/T_i)$ and $(Z/K,\bar{z},H/K)$. This completes the proof.
\end{proof}

\begin{lem}\label{dist_gap_lem}
	Let 
	$$(Y_i,y_i,G_i)\overset{GH}\longrightarrow (Z,z,H)$$ 
	be a convergent sequence of spaces in $\Omega(\widetilde{M},\mathcal{N})$. Let $(k_i,d_i)$ be the type of $(Y_i,y_i,G_i)$ and let $(k_\infty,d_\infty)$ be the type of $(Z,z,H)$. \\
	(1) Suppose that $k_i\ge k$ for all $i$, then $k_\infty\ge k$.\\
	(2) Suppose that $k_i=k$ and $d_i\ge 10$ for all $i$, then either $k_\infty>k$ holds, or $k_\infty=k$ and $d_\infty\ge 10$ hold.
\end{lem}

\begin{proof}
	(1) For each $i$, let $\{e_{i,1},...,e_{i,k_i}\}$ be an adapted basis for $(Y_i,y_i,G_i)$ as defined in Definition \ref{def_adapted_ver1}, where $k_i\ge k$. Let $T_i$ be the maximal torus subgroup of $G_i$ and let $L_{i,j}$ be the subgroup 
	$$L_{i,j}=\langle T_i,\mathbb{R}e_{i,1},...,\mathbb{R}e_{i,j} \rangle,$$
	where $j=1,...,k$. We remark that although the one-parameter subgroup $\mathbb{R}e_{i,j}$ is not unique, the above defined subgroup $L_{i,j}$ is uniquely defined and independent of the choice of the one-parameter subgroup through $e_j$. We consider the convergence
	$$(Y_i,y_i,G_i,T_i,L_{i,j})\overset{GH}\longrightarrow (Z,z,H,T_\infty,L_{\infty,j}).$$
	Let $(l_j,c_j)$ be the type of $(Z,z,L_{\infty,j})$. We prove $l_j\ge j$ by induction on $j$, then (1) follows by setting $j=k$.
	
	Let $j=1$. Note that the quotient group $L_{i,1}/T_i$ is a closed $\mathbb{R}$-subgroup of $\mathrm{Isom}(Y_i/T_i)$. Thus for each $\delta>0$, there is an element $g_{i}(\delta) \in L_{i,1}-T_i$ such that
	$$\delta= d(g_i(\delta)\cdot T_iy_i,T_i y_i)=d(g_i(\delta)\cdot y_i,T_iy_i).$$
	Passing this property to the limit, then for any $\delta>0$, we have an element $g_\infty(\delta)\in L_{\infty,1}-T_\infty$ such that
	$$\delta= d(g_\infty(\delta)\cdot z,T_\infty z).$$
	In particular, there is a closed $\mathbb{R}$-subgroup of $L_{\infty,1}$ that is outside $T_\infty$. Thus $l_1\ge 1$.
	
	Suppose that the statement holds for $j$. We consider $j+1$ next. The argument is similar to the case $j=1$. Since $L_{i,j+1}/L_{i,j}$ is a closed $\mathbb{R}$-subgroup of $\mathrm{Isom}(Y_i/L_{i,j})$, for each $\delta>0$ there is an element  $g_{i}(\delta) \in L_{i,j+1}-L_{i,j}$ such that
	$$\delta= d(g_i(\delta)\cdot y_i,L_{i,j} y_i).$$
	We pass this property to the limit. Hence there is 
	$g_\infty(\delta)\in L_{\infty,j+1}-L_{\infty,j}$ such that
	$$\delta= d(g_\infty(\delta)\cdot z,L_{\infty,j} z).$$
	This shows that there is a closed $\mathbb{R}$-subgroup of $L_{\infty,j+1}$ that is outside $L_{\infty,j}$. Together with the inductive assumption, we conclude that $l_{j+1}\ge j+1$. This completes the induction.
	
	(2) Assuming that $k_i=k$, $d_i\ge 10$ for all $i$, and $k_\infty=k$, we shall prove that $d_\infty\ge 10$. We follow the same notations as in the proof of (1); in particular, we have
	$$(Y_i,y_i,G_i,T_i,L_{i,j})\overset{GH}\longrightarrow (Z,z,H,T_\infty,L_{\infty,j}).$$
	We set $L_{\infty,0}:=T_\infty$ for convenience. From the proof of (1), we know that for each $j=1,...,k$, there is a closed $\mathbb{R}$-subgroup in $ L_{\infty,j}-L_{\infty,j-1}$. This implies that $T_\infty$ must be compact; otherwise, $T_\infty$ would contain a closed $\mathbb{R}$-subgroup and thus $k_\infty \ge k+1$, which violates with $k_\infty=k$. It follows that there is $D>0$ such that $d_i\le D$ for all $i$. Thus $$d_\infty = \lim_{i\to\infty} d_i \ge 10.$$
\end{proof}

We are ready to prove the distance gaps.

\begin{proof}[Proof of Propositions \ref{eGH_gap_1} and \ref{eGH_gap_2}] 
	The proofs of these two statements are similar. We argue by contradiction to prove them. 
	
	For Proposition \ref{eGH_gap_1}, suppose that there are two sequences $\{(Y_i,y_i,G_i)\}_i$ and  $\{(Y'_i,y'_i,G'_i)\}_i$ of spaces in $\Omega(\widetilde{M},\mathcal{N})$ with the conditions below:\\
	(1) each $(Y_i,y_i,G_i)$ has type $(k,d_i)$ with $d_i\le 1$;  \\
	(2) each $(Y'_i,y'_i,G'_i)$ has type $(k',d'_i)$ with $k'>k$;\\
	(3) $d_{GH}((Y_i,y_i,G_i),(Y'_i,y'_i,G'_i))\to 0$ as $i\to\infty$.\\
	After passing to some subsequences, we can assume that two sequences converge to the same limit $(Z,z,H)\in \Omega(\widetilde{M},\mathcal{N})$. We write $(k_\infty,d_\infty)$ as the type of $(Z,z,H)$. Condition (2) above and Lemma \ref{dist_gap_lem}(1) imply $k_\infty\ge k'>k$.
	Let $T_i$ be the maximal torus subgroup of $G_i$ and let $T_\infty\subseteq H$ be its limit. Since $\mathrm{diam}(T_iy_i)\le 1$, $T_\infty$ is compact. Thus the quotient space $(Z/T_\infty,\bar{z},H/T_\infty)$ has type $(k_\infty,\overline{d_\infty})$ with $\overline{d_\infty}\le d_\infty$. By Lemma \ref{mod_max_torus}, we have convergence
	$$(Y_i/T_i,\bar{y}_i,G_i/T_i)\overset{GH}\longrightarrow (Z/T_\infty,\bar{z},H/T_\infty).$$
	Let $F:[0,1]^{k_\infty} \to (H/T_\infty)\bar{z}$ be an adapted map constructed in Definition \ref{def_adapted_ver1}. $F$ is continuous and injective by Lemma \ref{adapted_map_inj_ver1}. According to Lemma \ref{conv_to_adapted_map}, there is a sequence of continuous maps $F_i:[0,1]^{k_\infty}\to (G_i/T_i)\bar{y}_i$ converges uniformly to $F$. Note that each $(G_i/T_i)\bar{y}_i$ is homeomorphic to $\mathbb{R}^k$ with $k<k_\infty$. By large fiber lemma, there are two sequences $\{a_i\}$ and $\{b_i\}$ in $[0,1]^{k_\infty}$ such that
	$$F_i(a_i)=F_i(b_i),\quad |a_i-b_i|\ge 1.$$
	By the uniform convergence of $F_i$ to $F$, we can find $a,b\in [0,1]^{k_\infty}$ such that
	$$F(a)=F(b),\quad |a-b|\ge 1.$$
	This contradicts the injectivity of $F$.
	
	For Proposition \ref{eGH_gap_2}, suppose that we have contradicting convergent sequences $\{(Y_i,y_i,G_i)\}_i$ and  $\{(Y'_i,y'_i,G'_i)\}_i$ such that\\
	(1) each $(Y_i,y_i,G_i)$ has type $(k,d_i)$ with $d_i\le 1$;  \\
	(2) each $(Y'_i,y'_i,G'_i)$ has type $(k,d'_i)$ with $d'_i\ge 10$;\\
	(3) $d_{GH}((Y_i,y_i,G_i),(Y'_i,y'_i,G'_i))\to 0$ as $i\to\infty$.\\
	After passing to some subsequences, we let $(Z,z,H)$ be their common limit, whose type is denoted as $(k_\infty,d_\infty)$. By Condition (2) above and Lemma \ref{dist_gap_lem}(2), either\\
	\textit{Case I.} $k_\infty>k$, or\\
	\textit{Case II.} $k_\infty=k$ and $d_\infty\ge 10$.\\
	Let $T_i$ be the maximal torus subgroup of $G_i$ and let $T_\infty\subseteq H$ be its limit. By Lemma \ref{mod_max_torus}, we have convergence
	$$(Y_i/T_i,\bar{y}_i,G_i/T_i)\overset{GH}\longrightarrow (Z/T_\infty,\bar{z},H/T_\infty).$$
	
	In Case I, we construct an adapted map $F:[0,1]^{k_\infty} \to (H/T_\infty)\bar{z}$ as in Definition \ref{def_adapted_ver1}. Then we use Lemma \ref{conv_to_adapted_map} to obtain $F_i:[0,1]^{k_\infty}\to (G_i/T_i)\bar{y}_i$ that converges uniformly to $F$, where $(G_i/T_i)\bar{y}_i$ is homeomorphic to $\mathbb{R}^k$ with $k<k_\infty$. Then a contradiction follows from the large fiber lemma and injectivity of $F$.
	
	In Case II, because $\mathrm{diam}(T_iy_i)\le 1$ for all $i$, the limit space $(Z/T_\infty,\bar{z},H/T_\infty)$ has type $(k,\bar{d})$ with $\bar{d}\ge 9$. Then we construct an adapted map $F:[0,1]^{k+1} \to (H/T_\infty)\bar{z}$ as in Definition \ref{def_adapted_ver2}. A similar contradiction arises since the targets of the approximated maps $F_i:[0,1]^{k+1}\to (G_i/T_i)\bar{y}_i$ are homeomorphic to $\mathbb{R}^k$.
\end{proof}

\section{Asymptotic orbits}\label{sec_cyclic}

This section studies the geometry of spaces in $\Omega(\widetilde{M},\mathcal{N})$ and $\Omega(\widetilde{M},\langle \gamma \rangle)$, where $\gamma \in \zeta_{l-1}(\mathcal{N})$. One of the main goals of this section is Proposition C(1), which states that $Hy$ is homeomorphic to $\mathbb{R}$ for all $(Y,y,H)\in \Omega(\widetilde{M},\langle \gamma \rangle)$. 

Some understandings of $(Y,y,G)\in \Omega(\widetilde{M},\mathcal{N})$ are required beforehand to prove Proposition C(1). In subsection \ref{subsec_whole_orbit}, we show that every $(Y,y,G)\in\Omega (\widetilde{M},\mathcal{N})$ is of type $(k_0,0)$ for some uniform $k_0$ (Proposition \ref{topol_dim}); in particular, we prove Theorem A(1). As mentioned in the introduction, the proof of Proposition \ref{topol_dim} uses the distance gaps in Section \ref{sec_egh_gap} and a critical rescaling argument. Subsection \ref{subsec_one_para_orbit} studies the one-parameter orbits of $Gy$ and makes the preparation for Proposition C(1). Lastly, we prove Proposition C(1) in subsection \ref{subsec_pf_c1}. The proof of Proposition C(1) follow a similar strategy as \cite[Section 3]{Pan_cone}: suppose that the statement fails, then we would find a suitable rescaling such that its limit space violates Proposition \ref{topol_dim}.

\subsection{Uniform type of asymptotic orbits and proof of Theorem A(1)}\label{subsec_whole_orbit}

\begin{prop}\label{topol_dim}
	Let $(M,p)$ be an open $n$-manifold with $\mathrm{Ric}\ge 0$ and $E(M,p)\not=1/2$. Let $\mathcal{N}\leq \pi_1(M)$ be a torsion-free nilpotent subgroup of finite index. Then there is an integer $k_0$ such that every $(Y,y,G)\in\Omega (\widetilde{M},\mathcal{N})$ is of type $(k_0,0)$. Consequently,\\
	(1) the orbit $Gy$ has a natural simply connected nilpotent group structure of dimension $k_0$;\\
	(2) any compact subgroup of $G$ fixes $y$;\\
	(3) $G$ has at most finitely many components.
\end{prop}

%\begin{rem}
%	We mention that by using structure result of nilpotent groups, Proposition \ref{topol_dim}(2) can be further improved: any compact subgroup of $G$ fixes $y$. The statement in its current form is sufficient for the rest of the paper.
%\end{rem}

We first prove a weaker statement without a uniform $k_0$.

\begin{lem}\label{width_zero}
	Under the assumptions of Proposition \ref{topol_dim}, let $(Y,y,G)\in \Omega(\widetilde{M},\mathcal{N})$. Then $(Y,y,G)$ is of type $(k,0)$ for some integer $k$.
\end{lem}

\begin{proof}
	We argue by contradiction. Suppose that the statement fails. Then we can choose $(Y,y,G)\in\Omega(\widetilde{M},\mathcal{N})$ with type $(k_0,d)$ such that\\
	(1) $d>0$, and\\
	(2) if another space $(Y',y',G')\in\Omega(\widetilde{M},\mathcal{N})$ is of type $(k',d')$ with $d'>0$, then $k'\ge k_0$.
	
	Let
	$$(Y_1,y_1,G_1)=(10d^{-1}Y,y,G),\quad (Y_2,y_2,G_2)=(d^{-1}Y,y,G)$$
	be two spaces in $\Omega(\widetilde{M},\mathcal{N})$. It is clear that they are of type $(k_0,10)$ and $(k_0,1)$, respectively. Let $r_i,s_i\to\infty$ such that
	$$(r_i^{-1}\widetilde{M},\tilde{p},\mathcal{N})\overset{GH}\longrightarrow (Y_1,y_1,G_1),\quad (s_i^{-1}\widetilde{M},\tilde{p},\mathcal{N})\overset{GH}\longrightarrow (Y_2,y_2,G_2).$$
	Passing to some subsequence, we can assume that $t_i:=r_i/s_i\to\infty$.
	Let 
	$$(N_i,q_i,\Gamma_i)=(r_i^{-1}\widetilde{M},\tilde{p},\mathcal{N}),$$ 
	then
	$$(N_i,q_i,\Gamma_i)\overset{GH}\longrightarrow (Y_1,y_1,G_1),\quad (t_iN_i,q_i,\Gamma_i)\overset{GH}\longrightarrow (Y_2,y_2,G_2).$$
	
	Let $\delta=\min\{\delta_1,\delta_2\}>0$, where $\delta_1$ and $\delta_2$ are the constants in Propositions \ref{eGH_gap_1} and \ref{eGH_gap_2}, respectively. For each $i$, we define a set of scales
	\begin{align*}
	L_i=\{ l\ge 1|&d_{GH}((lN_i,q_i,\Gamma),(W,w,H))\le \delta/10 \text{ for some} \\
	 &(W,w,H)\in\Omega(\widetilde{M},\mathcal{N}) \text{ such that } (W,w,H) \text{ has}\\
	 &\text{type $(k,d)$ with } k<k_0, \text{ or with } k=k_0 \text{ and } d\le 1 \}.
	\end{align*}
    Recall that $(Y_2,y_2,G_2)$ is of type $(k_0,1)$, thus $t_i\in L_i$ for all $i$ large. We choose $l_i\in L_i$ with $\inf L_i\le l_i \le \inf L_i+1/i$.
    
    \textbf{Claim 1:} $l_i\to\infty$. Suppose that $l_i\to l_\infty <\infty$ for a subsequence, then
    $$(l_iN_i,q_i,\Gamma_i)\overset{GH}\longrightarrow (l_\infty Y_1,y_1,G_1).$$
    Since $l_i\in L_i$, for each $i$ there is some $(W_i,w_i,H_i)$ with the properties in the definition of $L_i$ such that
    $$d_{GH}((l_iN_i,q_i,\Gamma_i),(W_i,w_i,H_i))\le \delta/10.$$
    Hence for $i$ large,
    $$d_{GH}((W_i,w_i,H_i),(l_\infty Y_1,y_1,G_1))\le \delta/2,$$
    where $(l_\infty Y_1,y_1,G_1)$ is of type $(k_0,10l_\infty)$ with $10l_\infty\ge 10$. Let $(k_i,d_i)$ be the type of $(W_i,w_i,H_i)$. If $k_i<k_0$, then $d_i=0$ by our choice of $k_0$, and the above Gromov-Hausdorff distance estimate cannot hold due to Proposition \ref{eGH_gap_1} and the choice of $\delta$. If $k_i=k_0$ and $d_i\le 1$, then it also leads to a contradiction due to Proposition \ref{eGH_gap_2}. Therefore, $l_i\to\infty$.
    
    Next, we consider the convergence
    $$(l_iN_i,q_i,\Gamma_i)\overset{GH}\longrightarrow (Y',y',G')\in\Omega(\widetilde{M},\mathcal{N}).$$
    Let $(k',d')$ be the type of $(Y',y',G')$.
    
    \textbf{Claim 2:} $k'\le k_0$; moreover, $d'< 10$ when $k'=k_0$. For each $i$, there is some $(W_i,w_i,H_i)$ with the properties in the definition of $L_i$ and
    $$d_{GH}((l_iN_i,q_i,\Gamma_i),(W_i,w_i,H_i))\le \delta/10.$$ 
    It follows that for $i$ large,
    $$d_{GH}((W_i,w_i,H_i),(Y',y',G'))\le\delta/2.$$
    If $k'>k_0$, then we end in a contradiction to Proposition \ref{eGH_gap_1}. If $k'=k_0$ and $d'>10$, then this contradicts with Proposition \ref{eGH_gap_2}. This proves Claim 2.
    
    By Claim 2 and our choice of $k_0$, $(Y',y',G')$ has type $(k',d')$ with one of the following cases:\\
    \textit{Case 1:} $k'<k_0$ and $d'=0$;\\
    \textit{Case 2:} $k'=k_0$ and $d'< 10$.\\
    For Case 1, we consider the convergence sequence
    $$(\frac{1}{2}l_iN_i,q_i,\Gamma_i)\overset{GH}\longrightarrow (\frac{1}{2}Y',y',G')\in\Omega(\widetilde{M},\mathcal{N}).$$
    The limit space $(\frac{1}{2}Y',y',G')$ is of type $(k',0)$, where $k'<k_0$. This implies that $l_i/2\in L_i$, a contradiction to $\inf L_i\le l_i\le \inf L_i+1/i$. For Case 2, we consider 
    $$(\frac{1}{10}l_iN_i,q_i,\Gamma_i)\overset{GH}\longrightarrow (\frac{1}{10}Y',y',G')\in\Omega(\widetilde{M},\mathcal{N}),$$
    where $(\frac{1}{10}Y',y',G')$ is of type $(k',d'/10)$ with $d'/10< 1$. We result in $l_i/10\in L_i$ for $i$ large and thus a desired contradiction.
    
    With all possibilities of $(Y',y',G')$ being ruled out, we complete the proof.
\end{proof}

Next, we prove Proposition \ref{topol_dim}.

\begin{proof}[Proof of Proposition \ref{topol_dim}]
	Let $(Y,y,G)$ and $(Y',y',G')$ in $\Omega(\widetilde{M},\mathcal{N})$ having type $(k,0)$ and $(k',0)$, respectively. We show that $k= k'$. Let $\delta_1$ be the constant in Proposition \ref{eGH_gap_1} and let $\epsilon=\delta_1/2$. Because the set $\Omega(\widetilde{M},\mathcal{N})$ is connected in the pointed equivariant Gromov-Hausdorff topology (Proposition \ref{cpt_cnt}), for the above $\epsilon>0$, there is a chain of elements $\{(W_j,w_j,H_j)\}_{j=1}^J$ in $\Omega(\widetilde{M},\mathcal{N})$ such that
	$$(W_1,w_1,H_1)=(Y,y,G),\quad (W_J,w_J,H_J)=(Y',y',G'),$$
	and
	$$d_{GH}((W_j,w_j,H_j),(W_{j+1},w_{j+1},H_{j+1}))\le \epsilon$$
	for all $j=1,...,J-1$.
	Applying Lemma \ref{width_zero}, we know that each $(W_j,w_j,H_j)$ is of type $(k_j,0)$ for some $k_j$. By Proposition \ref{eGH_gap_1} and our choice of $\epsilon$, all $k_j$ must be the same. In particular, we conclude $k=k'$.
	
	Now we proceed to prove the consequences (1)-(3).
	
	To prove (1), let $T$ be the maximal torus subgroup of $G_0$. Because $(Y,y,G)$ is of type $(k_0,0)$, $T$-action fixes $y$. Thus the orbit $Gy=G_0y$ can be naturally identified with the quotient group $G_0/T$, which is a simply connected nilpotent Lie group of dimension $k_0$.
	
	Next, we prove (2). Let $K$ be a compact subgroup of $G$. If $K$ is contained in a torus subgroup of $G$, then it is clearly that $Ky=y$ because $(Y,y,G)$ is of type $(k,0)$. In general, suppose that (2) fails, then we can find a finite cyclic subgroup $\langle h \rangle \subseteq K$ such that $hy\not= y$. Let $g\in G_0$ such that $hy=gy$. Note that $g$ is outside the maximal torus $T$ of $G_0$ because $T$ fixes $y$; in particular, $\langle g \rangle y$ is unbounded in $Y$. On the other hand, because $hy=gy$ and Lemma \ref{nil_G0_commute_cpt}, we have $h^k y=g^k y$ for all $k\in \mathbb{Z}$. Thus the set $\{g^k y| k\in\mathbb{Z}\}$ is bounded because $\langle h \rangle$ is a finite group. A contradiction.
	
	Lastly, we prove (3). Suppose that $G$ has infinitely many connected components. We consider the continuous map
	$$A: G \to Gy, \quad g \mapsto gy.$$
	Let $\mathcal{C}$ be any connected component of $G$. Because the orbit $Gy$ is connected, we have $A(\mathcal{C})=Gy$. Thus we can choose $g_{\mathcal{C}}\in \mathcal{C}$ such that $g_{\mathcal{C}}\cdot y=y$. By hypothesis, this gives a set of infinitely many elements 
	$$\mathcal{F}=\{g_\mathcal{C}\ |\  \mathcal{C} \text{ is a connected component of }G\}$$
	that fixes $y$. Since the isotropy subgroup of $G$ at $y$ is compact, we can find a convergence subsequence from $\mathcal{F}$. This contradicts with the fact that elements in $\mathcal{F}$ belong to distinct connected components of $G$.
\end{proof}

We complete this subsection by proving Theorem A(1).

\begin{proof}[Proof of Theorem A(1)]
	Let $\mathcal{N}$ be a torsion-free nilpotent subgroup of $\Gamma=\pi_1(M,p)$ with finite index. Let $r_i\to\infty$ and consider the convergence
	$$(r_i^{-1}\widetilde{M},\tilde{p},\mathcal{N},\Gamma)\overset{GH}\longrightarrow (Y,y,G,\overline{G}).$$
	We have already shown in Proposition \ref{topol_dim} that the orbit $Gy$ has a natural simply connected nilpotent group structure. Given that $Gy=G_0 y = G_0/T$ as explained in the proof of Proposition \ref{topol_dim}(1), it is also clear that
	$$\step(Gy)=\step(G_0/T)\le \step(G) \le \step(\mathcal{N})=\step(\pi_1(M)).$$
	
	It remains to show that $Gy=\overline{G}y$, which is standard. We give a proof for readers' convenience. Since $\mathcal{N}$ has finite index in $\Gamma$, we can write $\Gamma=\cup_{j=0}^k \mathcal{N} \gamma_j$ as union of cosets for some elements $\gamma_0=e,\gamma_1,...,\gamma_k$. Passing to a subsequence, we assume that
	$$(r_i^{-1}\widetilde{M},\tilde{p},\gamma_j)\overset{GH}\longrightarrow (Y,y,g_j),$$
	where $g_j \in \overline{G}$ fixing $y$. For any $h\in \overline{G}-G$, let $\alpha_i \in \Gamma$ be a sequence converging to $h$. We write $\alpha_i=\beta_i\cdot \gamma_{j(i)}$, where $\beta_i\in \mathcal{N}$, and then pass to a subsequence and assume $j(i)=b$ is a constant. Then $\beta_i=\alpha_i \cdot \gamma_b^{-1}$ has limit $hg_b^{-1}\in {G}$ by construction. Note that
	$$ hy= hg_b^{-1}(g_b y)=hg_b^{-1} y \in Gy.$$
	We conclude $Gy=\overline{G}y$ and hence complete the proof of Theorem A(1).
\end{proof}

\subsection{One-parameter orbits}\label{subsec_one_para_orbit}

\begin{defn}\label{def_one_para_orbit}
	Let $(Y,y,G)\in\Omega(\widetilde{M},\mathcal{N})$. Given an orbit point $z\in Gy-\{y\}$, we define the \textit{one-parameter orbit} through $z$ as the one-parameter subgroup of $Gy$ through $z$, where the simply connected nilpotent Lie group structure of $Gy$ is given in Proposition \ref{topol_dim}(1). By Lemma \ref{nil_exp}, this is well-defined. We denote this one-parameter orbit through $z$ as $\mathbb{R}z$, and the points on it as $tz$, where $t\in\mathbb{R}$.
\end{defn}

As indicated in its name, a one-parameter orbit is indeed the orbit of a one-parameter subgroup. More precisely, we have the description below.

\begin{lem}\label{one_para_orbit}
	Let $(Y,y,G)\in\Omega(\widetilde{M},\mathcal{N})$ and let $z\in Gy-\{y\}$ be an orbit point. We write $z=gy$, where $g\in G_0$. Let $\sigma:\mathbb{R}\to G_0$ be one-parameter subgroup with $\sigma(1)=g$. Then $\sigma(t)y=tz$ for all $t\in\mathbb{R}$.
\end{lem}	

\begin{proof}
	Let $T$ be the maximal torus subgroup of $G_0$. We naturally identify the orbit $Gy$ with the quotient group $G_0/T$ as in the proof of Proposition \ref{topol_dim}(1). Let $\pi:G_0 \to G_0/T$ be the quotient map and let $\bar{g}=\pi(g)$. Because $G_0/T$ is connected and simply connected, by Lemma \ref{nil_exp}, there is a unique one-parameter subgroup $\bar{\sigma}:\mathbb{R}\to G_0/T$ such that $\bar{\sigma}(1)=\bar{g}$. Since $\pi$ is a group homomorphism, we have $$\sigma(t)y=\pi\circ \sigma(t)=\bar{\sigma}(t).$$
	Because the choice of $\bar{g}\in G_0/T$ and $\bar{\sigma}$ are uniquely determined by the orbit point $z=gy$, the result follows.
\end{proof}

\begin{cor}\label{one_para_trans}
	Let $(Y,y,G)\in\Omega(\widetilde{M},\mathcal{N})$ and let $z\in Gy-\{y\}$ be an orbit point. Let $g\in G_0$ such that $z=gy$. Then $g\cdot (tz)=(1+t)z\in \mathbb{R}z$ for all $t\in\mathbb{R}$.
\end{cor}

\begin{proof}
	Let $\sigma:\mathbb{R}\to G_0$ be a one-parameter subgroup of $G_0$ with $\sigma(1)=g$. By Lemma \ref{one_para_orbit}, for any $t\in\mathbb{R}$, we have
	$$g\cdot (tz)=g\cdot \sigma(t)y=\sigma(1+t)y=(1+t)z.$$
	The result follows.
\end{proof}

Given $\gamma\in \mathcal{N}-\{\mathrm{id}\}$ and $k\in\mathbb{Z}$, we denote
$$S(\gamma,k):=\{\gamma^m\ |\ m=0,\pm1,...,\pm k\},$$
which is a symmetric subset of $\langle \gamma \rangle$. Recall that $\mathcal{N}$ is torsion-free, thus $\gamma$ has infinite order. Below, we will use the notion of convergence of symmetric subsets; see Definition \ref{def_conv_symsubset}.

\begin{lem}\label{reduction_id}
	Let $\gamma\in \mathcal{N}-\{\mathrm{id}\}$. We consider the convergence
	$$(r_i^{-1}\widetilde{M},\tilde{p},\gamma)\overset{GH}\longrightarrow (Y,y,g_0).$$
	Then there is a non-decreasing sequence $\{t_i\}$ in $\mathbb{Z}_+$ such that after passing to a subsequence, we have
	$$(r_i^{-1}\widetilde{M},\tilde{p},\gamma^{t_i},S(\gamma,t_i))\overset{GH}\longrightarrow (Y,y,\mathrm{id},S)$$
	with $Sy=y$.
\end{lem}

\begin{proof}
	Note that by construction $g_0$ fixes the base point $y$. If $g_0$ has finite order, we let $t\in \mathbb{Z}_+$ such that $g_0^t=\mathrm{id}$. Then it holds that
	$$(r_i^{-1}\widetilde{M},\tilde{p},\gamma^{t},S(\gamma,t))\overset{GH}\longrightarrow (Y,y,\mathrm{id},S)$$
	with $S$ fixing $y$.
	
	It remains to consider the case that $g_0$ has infinite order. Since $g_0$ fixes $y$, $\langle g_0 \rangle$ is a precompact subgroup of the isotropy subgroup at $y$. Let $t_j\to \infty$ be a sequence such that $g_0^{t_j}\to \mathrm{id}$ as $j\to\infty$. By a standard diagonal argument, we can choose a subsequence and $t_{i(j)}\to\infty$ such that
	$$(r_{i(j)}^{-1}\widetilde{M},\tilde{p},\gamma^{t_{i(j)}},S(\gamma,t_{i(j)}))\overset{GH}\longrightarrow (Y,y,\mathrm{id},\overline{\langle g_0 \rangle}).$$
	The result follows.
\end{proof}

\begin{lem}\label{limit_one_para}
	Let $\gamma\in \mathcal{N}-\{\mathrm{id}\}$ and $r_i\to\infty$. We choose
	$t_i$ as in Lemma \ref{reduction_id} and	
	$$k_i:=\min\{l \in\mathbb{Z}_+\ |\ d(\gamma^{t_i l} \tilde{p},\tilde{p}) \ge r_i\}\to\infty.$$
	Passing to a subsequence, we consider the convergence
	$$(r_i^{-1}\widetilde{M},\tilde{p},\gamma^{t_i k_i},S(\gamma^{t_i},k_i),\langle\gamma \rangle,\mathcal{N})\overset{GH}\longrightarrow (Y,y,g,A,H,G).$$
	We denote $z=gy$. Then\\
	(1) $A$ is connected, in particular $A\subseteq G_0$;\\
	(2) $Ay$ contains the set $\{tz|t\in[0,1]\}$;\\
	(3) $Hy$ contains the one-parameter orbit $\mathbb{R}z$; moreover, there is a one-parameter subgroup $L$ of $H_0$ such that $Ly=\mathbb{R}$z.
\end{lem}

\begin{proof}
	(1) First note that by the choice of $k_i$ and the construction of $A$, it is clear that $Ay\subseteq \overline{B_1}(y)$.
	
	We argue by contradiction to prove that $A$ is connected. Suppose otherwise, then by Proposition \ref{G0_by_orbit} there is a point $w\in Y$ such that $Aw$ has multiple connected components. Let $\alpha w\in Aw$ be a point outside the connected component of $Aw$ that contains $w$. We choose a sequence of point $q_i\in\widetilde{M}$ and a sequence of integers $m_i$ within $[-k_i,k_i]$ such that
	$$(r_i^{-1}\widetilde{M},\tilde{p},q_i,\gamma^{t_im_i})\overset{GH}\longrightarrow (Y,y,w,\alpha).$$ 
	Replacing $\alpha$ by $\alpha^{-1}$ if necessary, we can assume that $m_i>0$. Recall that $t_i$ is chosen in Lemma \ref{reduction_id} such that $\gamma^{t_i}$ converges to identity under the convergence; in particular, we have
	$$r_i^{-1}d(\gamma^{t_i}q_i,q_i)\to 0$$
	as $i\to\infty$. Since $\alpha$ moves $w$, we must have $m_i\to\infty$.
	
	Let $\epsilon\ll d(w,\alpha w)$. For each sufficiently large $i$, we choose an integer $s_i$ within $[1,m_i]$ such that
	$$  \epsilon/2 \le r_i^{-1}d(\gamma^{t_is_i}q_i,q_i) \le \epsilon.$$
	We consider the symmetric subset 
	$$T_i=\langle \gamma^{t_is_i} \rangle \cap S(\gamma^{t_i}, m_i)$$
	and its convergence
	$$(r_i^{-1}\widetilde{M},\tilde{p},q_i,T_i)\overset{GH}\longrightarrow (Y,y,q,T).$$
	By construction and the compactness of $Tw\subseteq Aw \subseteq \overline{B_{1+2D}}(y)$, where $D=d(w,y)$, we can find an $\epsilon$-chain $\{y_0=y,y_1,....,y_N=\alpha w\}$ in $Tw\cup \{\alpha w\}\subseteq Aw$ such that
	$$d(y_j,y_{j+1})\le\epsilon$$
	for all $j$. Because the small $\epsilon>0$ is arbitrary, this shows that $\alpha w$ belongs to the same connected component of $Aw$ as $w$; a contradiction.
	
	(2) Note that by the choice of $k_i$, it holds that
	$$r_i \le d(\gamma^{t_i k_i}\tilde{p},\tilde{p})\le r_i+d(\gamma^{t_i}\tilde{p},\tilde{p})$$
	with $r_i^{-1}d(\gamma^{t_i}\tilde{p},\tilde{p})\to 0$. Thus $d(gy,y)=1$.
	By (1), we have $g\in A \subseteq G_0$.
	
	Below we write $\delta_i=\gamma^{t_i}$ for convenience. Let $b\ge 2$ be an integer. Due to the choice of $k_i$, we have
	$$d(\delta_i^{\lceil k_i/b \rceil}\tilde{p},\tilde{p})<r_i,$$
	where $\lceil\cdot\rceil$ is the ceiling function. After passing to a subsequence, we can assume that $$b\cdot \lceil k_i/b\rceil=k_i+b_0$$
	for all $i$, where $b_0\in\{0,1,...,b-1\}$, and the convergence
	$$(r_i^{-1}\widetilde{M},\tilde{p},\delta_i^{\lceil k_i/b \rceil},\delta_i^{b_0},\delta_i^{b\cdot \lceil k_i/b\rceil})\overset{GH}\longrightarrow (Y,y,\beta,\mathrm{id},g).$$
	By construction, we have $\beta^b=g;$ moreover, $\beta\in A\subseteq G_0$. Let $\sigma:\mathbb{R}\to G_0$ be a one-parameter subgroup of $G_0$ with $\sigma(1)=\beta$. Then
	$$\eta:\mathbb{R}\to G_0\quad t\mapsto \sigma(bt)$$
	defines a one-parameter subgroup of $G_0$ with $\eta(1)=\sigma(b)=g$. By Lemma \ref{one_para_orbit}, we have
	$$\sigma(bt)y=\eta(t)y=tz$$ 
	for all $t\in\mathbb{R}$, where $z=gy$. 
	
	We set $t=m/b$, where $m\in \{0,\pm1,...,\pm (b-1)\}$. Then we have
	$$\frac{m}{b}z=\sigma(m)y=\beta^m y.$$
	Note that $\beta^m$ is the limit of $\delta_i^{m\lceil k_i/b \rceil}\in S(\delta_i,k_i)$. Thus $\frac{m}{b}z \in Ay$ for all $b\ge 2$ and all $m\in \{0,\pm1,...,\pm (b-1)\}$. We conclude that $Ay$ must contain $\{tz|t\in[0,1]\}$.
	
	(3) Using the notations in the proof of (2), we choose $t=m/b$, where $m\in\mathbb{Z}$. Then similarly, we have
	$$\frac{m}{b}z=\beta^m y.$$
	Since $\langle \beta\rangle \in H$, we conclude that the orbit $Hy$ contains the points
	$$\left\{\frac{m}{b}z|m\in\mathbb{Z}\right\}.$$
	Because the integer $b\ge 2$ is arbitrary, $Hy$ must contain $\mathbb{R}z$.
	
	Lastly, we show that $\mathbb{R}z=Ly$ for a one-parameter subgroup $L$ of $H_0$. In fact, let $\mathcal{C}$ be the connected component of $Hy$ that contains $y$. Note that 
	$$\mathbb{R}z\subseteq \mathcal{C} \subseteq H_0y.$$
	Now the result follows from Lemma \ref{one_para_orbit}.
\end{proof}

\begin{rem}\label{rem_sym_one_para}
	In the context of Lemma \ref{limit_one_para}, if a sequence $n_i$ satisfies $n_i\gg {t_i k_i}$, then with respect to the convergence
	$$(r_i^{-1}\widetilde{M},\tilde{p},S(\gamma,n_i))\overset{GH}\longrightarrow (Y,y,B),$$
	$By$ must contain $\mathbb{R}z$. In fact, from the proof of Lemma \ref{limit_one_para}(2), we have 
	$$(r_i^{-1}\widetilde{M},\tilde{p},\delta_i^{\lceil k_i/b \rceil})\overset{GH}\longrightarrow (Y,y,\beta),\quad \frac{m}{b}z=\beta^m y,$$
	where $b\in\mathbb{Z}\cap [2,\infty)$ and $m\in\mathbb{Z}$.
    Together with $n_i\gg t_ik_i$,  $\frac{m}{b}z$ is the limit of $\delta_i^{m\lceil k_i/b \rceil}\cdot \tilde{p}\in S(\gamma,n_i)\cdot \tilde{p}$. It follows that $By$ contains $\mathbb{R}z$.
\end{rem}

\subsection{Proof of Proposition C(1)}\label{subsec_pf_c1}

We prove Proposition C(1) in this subsection. For convenience, we restate it as follows by using the terminology from Section \ref{subsec_one_para_orbit}.

\begin{prop}\label{c1}
		Let $(M,p)$ be an open $n$-manifold with $\mathrm{Ric}\ge 0$ and $E(M,p)\not= 1/2$. Let $\mathcal{N}$ be a torsion-free nilpotent subgroup of $\pi_1(M,p)$ with finite index and let $l$ be the nilpotency step of $\mathcal{N}$. 
		
		Then for every $\gamma\in \zeta_{l-1}(\mathcal{N})-\{\mathrm{id}\}$ and every $(Y,y,H,G)\in \Omega(\widetilde{M},\langle \gamma \rangle,\mathcal{N})$, the orbit $Hy$ is exactly the one-parameter orbit $\mathbb{R}z$, where $z\in Hy-\{y\}$. Moreover, the orbit $Hy$ can be represented by $\{\sigma(t)y|t\in \mathbb{R}\}$, where $\sigma:\mathbb{R}\to H_0 \subseteq Z(G)$ is a one-parameter subgroup through $g$ with $gy\in Hy$.
\end{prop}

We start with two corollaries below which follow directly from Proposition \ref{topol_dim}(2).

\begin{cor}\label{orbit_power_back}
	Let $(Y,y,H)\in\Omega(\widetilde{M},\langle \gamma \rangle)$ and let $h_1,h_2\in H$. Suppose that there is an integer $m\ge 2$ such that $h_1^m y=h_2^m y$, then $h_1y=h_2y$.
\end{cor}

\begin{proof}
	Let $r_i\to\infty$ such that
	$$(r_i^{-1}\widetilde{M},\tilde{p},\langle \gamma \rangle, \mathcal{N})\overset{GH}\longrightarrow (Y,y,H,G).$$
	
	We first claim that if an element $h\in H$ satisfies $h^m y=y$ for some integer $m\ge 2$, then $hy=y$. In fact, let $H\le G$ be the closure of the subgroup generated by $h$. By assumption, its orbit $Hy$ consists of at most $m$ many elements. In particular, $H$ is compact. By Proposition \ref{topol_dim}(2), we conclude $Hy=y$.
	
	Now let $h_1, h_2\in H$ such that $h_1^my=h_2^m y$. Because $H$ is abelian, we have 
	$$(h_2^{-1}h_1)^m y=h_2^{-m} h_1^m y=y.$$
	It follows from the Claim that $h_2^{-1}h_1y=y$.
\end{proof}

\begin{cor}\label{bdd_sym_fix}
	Let $(Y,y,H)\in\Omega(\widetilde{M},\langle \gamma \rangle)$ and let $S$ be a closed symmetric subset of $H$. Suppose that the set $Sy=\{sy|s\in S\}$ satisfies the following:\\
	(1) $Sy$ is closed under multiplication, that is, if $s_1,s_2\in S$, then $s_1s_2y\in Sy$;\\
	(2) $Sy$ is bounded.\\
	Then $Sy=\{y\}$.
\end{cor}

\begin{proof}
	Let $L={\langle S \rangle}$, the subgroup generated by $S$. By assumptions, $Ly=Sy$ is bounded. Thus the closure $\overline{L}$ is a compact subgroup of $H$. It follows from Proposition \ref{topol_dim}(2) that $\overline{L}$ must fix $y$. Thus $Sy=Ly=\{y\}$.
\end{proof}

We are ready to prove Proposition \ref{c1}.

\begin{proof}[Proof of Proposition \ref{c1}]
	From Lemma \ref{limit_one_para}, we know that the limit orbit $Hy$ from
	$$(r_i^{-1}\widetilde{M},\tilde{p},\langle \gamma \rangle,\mathcal{N})\overset{GH}\longrightarrow (Y,y,H,G)$$
	contains a one-parameter orbit $\mathbb{R}z$, where $z=gy$ is constructed in Lemma \ref{limit_one_para}. Below, we continue to use the notations from Lemma \ref{limit_one_para}; in particular, we have $$(r_i^{-1}\widetilde{M},\tilde{p}),\gamma^{t_i k_i},S(\gamma^{t_i},k_i))\overset{GH}\longrightarrow (Y,y,g,A)$$ 
	with 
	$$g\in A\subseteq G_0,\quad d(gy,y)=1,\quad Ay\subseteq \overline{B_1}(y),$$
	where the sequences $t_i$ and $k_i$ are described in Lemmas \ref{reduction_id} and \ref{limit_one_para}. 
	
	We argue by contradiction to show that $Hy=\mathbb{R}z$. Suppose that there is $h\in H$ such that $hy\not\in \mathbb{R}z$.
	
	\textbf{Claim 1:} Without lose of generality, we can assume that $d(hy,\mathbb{R}z)\ge 2$.
	
	The element $h$ may not be in $G_0$. However, because $G$ has at most finitely many components (see Proposition \ref{topol_dim}(3)), we can find a power $n\in\mathbb{Z}_+$ such that $h^n\in G_0$. For this $h^n$, we still have the property that $h^n y \notin \mathbb{R}z$. In fact, suppose that $h^n y\in \mathbb{R}z$. By Lemma \ref{limit_one_para}(3), we have a one parameter subgroup $\eta:\mathbb{R}\to H_0$ such that $\eta(1)=g_0\in H_0$ and $g_0y=h^n y$. Then by Corollary \ref{orbit_power_back}, we see that $hy=\eta(1/n)y\in\mathbb{R}z$; a contradiction. Now that we have $h^n\in H\cap G_0$ with $h^n y\notin \mathbb{R}z$; next, we show that 
	$d((h^n)^m y,\mathbb{R}z)$ is unbounded as $m\to \infty$. In fact, let $\sigma:\mathbb{R}\to G_0$ be a one-parameter subgroup such that $\sigma(1)y=z$. Let $\overline{L}$ be the subgroup generated by elements in $\sigma$ and $T$, the maximal torus subgroup of $G_0$. Because $T$ is central in $G_0$, each element $\overline{L}$ can be expressed as $\sigma(t)\cdot \theta$, where $\theta\in T$. Moreover, $\overline{L}y=\mathbb{R}z$ because $T$ fixes $y$ according to Proposition \ref{topol_dim}; in particular, $h^n\notin \overline{L}$. By construction, the quotient group $G_0/\overline{L}$ is a connected and simply connected nilpotent Lie group and $G_0/\overline{L}$ acts on the quotient space $(Y/\overline{L},\bar{y})$. Let $q:G_0\to G_0/\overline{L}$ be the quotient homomorphism, then $q(h^n)$ generates a discrete $\mathbb{Z}$-subgroup in $G_0/\overline{L}$. Thus $d(q(h^n)^m\bar{y},\bar{y})$ is unbounded as $m\to\infty$. As a consequence, $d((h^n)^m y,\mathbb{R}z)$ is unbounded as $m\to \infty$. To this end, we choose $m$ such that $d((h^n)^m y,\mathbb{R}z)\ge 2$. Replacing $h$ by $h^{nm}$, we complete Claim 1.
	
	Let $m_i\in\mathbb{Z}$ such that
	$$(r_i^{-1}\widetilde{M},\tilde{p},\gamma^{m_i})\overset{GH}\longrightarrow (Y,y,h).$$
	Replacing $h$ by $h^{-1}$ if necessary, we can assume that $m_i>0$.
	
	\textbf{Claim 2:} $m_i \gg t_i k_i$.
	
	By $d(hy,y)\ge 2$ and the choice of $k_i$, we clearly have $m_i>t_ik_i$. To prove Claim 2, suppose that $m_i/(t_ik_i)\to C<\infty$ for a subsequence, then we can write
	$$m_i=\lfloor C \rfloor t_ik_i +o_i,$$
	where $\lfloor \cdot \rfloor$ is the floor function and $o_i\in \mathbb{Z}\cap [0,t_ik_i]$. Passing to a subsequence, we have convergence
	$$(r_i^{-1}\widetilde{M},\tilde{p},\gamma^{\lfloor C \rfloor t_ik_i},\gamma^{o_i},\gamma^{m_i})\overset{GH}\longrightarrow (Y,y,g^{\lfloor C \rfloor},\delta,h),$$
	where $\delta\in A$. Since $g\in G_0$, by Corollary \ref{one_para_trans}, we have
	$$g^{\lfloor C \rfloor}\cdot(tz)=(\lfloor C \rfloor+t) z\in\mathbb{R}z$$
	for all $t\in\mathbb{R}$. Consequently, 
	$$d(hy,\mathbb{R}z)=d(g^{\lfloor C \rfloor}\delta y,\mathbb{R}z)=d(\delta y,\mathbb{R}z)\le 1;$$
	A contradiction to  $d(hy,\mathbb{R}z)\ge 2$. This proves Claim 2.
	
	For each $i$, we define
	$$d_i=\max\{d(\gamma^k\tilde{p},\tilde{p})\ |\ k\in \mathbb{Z}\cap [t_ik_i,m_i] \}.$$
	
	\textbf{Claim 3:} $d_i \gg r_i$.
	
	It is clear that $d_i\ge r_i$. Suppose that $d_i/r_i\to C<\infty$ for a subsequence. Then we consider the convergence
	$$(d_i^{-1}\widetilde{M},\tilde{p},S(\gamma,m_i))\overset{GH}\longrightarrow (C^{-1}Y,y,B).$$
	By the proof of Lemma \ref{limit_one_para} and $m_i\gg t_ik_i$, $By$ must contain $\mathbb{R}z$ (see Remark \ref{rem_sym_one_para}). Hence $By$ is unbounded. On the other hand, by the choice of $d_i$, we should have $By\subseteq \overline{B_1}(y)$; a contradiction. This proves Claim 3.
	
	Next, we consider the blow-down under $d_i\to\infty$:
	$$(d_i^{-1}\widetilde{M},\tilde{p},\gamma^{m_i},S(\gamma,m_i),\langle \gamma \rangle)\overset{GH}\longrightarrow (Y',y',h',B',H').$$
	By the choice of $d_i$, we have $B'y'\subseteq \overline{B_1}(y')$. Also, note that Claim 3 implies that $h'y'=y'$.
	
	\textbf{Claim 4:} $B'y'$ is closed under multiplication.
	
	Let $\beta_1,\beta_2\in B'$. We shall show $\beta_1\beta_2y'\in B'y'$. We choose $b_{1,i},b_{2,i}\in \mathbb{Z}\cap[-m_i,m_i]$ such that
	$$(d_i^{-1}\widetilde{M},\tilde{p},\gamma^{b_{1,i}},\gamma^{b_{2,i}})\overset{GH}\longrightarrow (Y',y',\beta_1,\beta_2).$$
	Then $\beta_1\beta_2$ is the limit of $\gamma^{b_{1,i}+b_{2,i}}$. If $b_{1,i}+b_{2,i}\in [-m_i,m_i]$, then $\beta_1\beta_2\in B'$ clearly holds. If not, we write
	$$b_{1,i}+b_{2,i}=\pm m_i + o_i,$$
	where $o_i\in \mathbb{Z}\cap[-m_i,m_i]$. Passing to a subsequence if necessary, we have
	$$(d_i^{-1}\widetilde{M},\tilde{p},\gamma^{m_i},\gamma^{o_i})\overset{GH}\longrightarrow (Y',y',h',\beta_0),$$
	where $\beta_0\in B'$. Then with respect to the blow-down sequence of $\widetilde{M}$ by $d_i^{-1}$, we have
	$$\beta_1\beta_2y'=\lim \gamma^{o_i}\gamma^{\pm m_i} \tilde{p}=\beta_0(h')^{\pm 1}y'=\beta_0 y'\in B'y'.$$
	This proves Claim 4.
	
	Lastly, by Claim 4 and Corollary \ref{bdd_sym_fix}, we conclude that $B'y'=\{y'\}$. On the other hand, the choice of $d_i$ implies that $d_H(B'y',y')=1$. This contradiction shows that $Hy=\mathbb{R}z$ and thus completes the proof of Proposition \ref{c1}.
\end{proof}

We complete this section by proving a distance control on the one-parameter orbit.

\begin{lem}\label{boomer_control}
	Under the assumptions of Proposition C, there is a constant $C_1=C_1(\widetilde{M},\gamma)$ such that for any $(Y,y,H)\in\Omega(\widetilde{M},\langle\gamma \rangle)$ and any orbit point $z\in Hy-\{y\}$, we have
	$$d(tz,y)\le C_1\cdot d(z,y)$$
	for all $t\in [0,1]$.
\end{lem}

\begin{proof}
	Scaling $(Y,y,H)$ by a constant, we may assume that $d(z,y)=1$. Recall that by Proposition \ref{c1} we can choose $h\in H_0$ such that $z=hy$. We argue by contradiction and suppose that there are contradicting sequences $(Y_j,y_j,H_j)\in \Omega(\widetilde{M},\langle \gamma \rangle)$ and $h_j\in (H_j)_0$ with $d(h_jy_j,y_j)=1$ but
	$$R_j:=\max_{t\in [0,1]} d(tz_j,y_j)\to\infty$$
	as $j\to\infty$, where $z_j=h_jy_j\in Y_j$. For each $j$, we choose a one-parameter subgroup of $(H_j)_0$ through $h_j$ and use $t h_j$ to denote elements in the subgroup, where $t\in\mathbb{R}$. By Lemma \ref{one_para_orbit}, $(th_j)y_j=tz_j$ for all $t\in\mathbb{R}$. We consider symmetric subsets
	$$S_j=\{ th_j |t\in[-1,1]\}$$
	and the convergence
	$$(R_j^{-1}Y_j,y_j,H_j,S_j)\overset{GH}\longrightarrow (Y',y',H',S'),$$
	where $(Y',y',H')\in\Omega(\widetilde{M},\langle\gamma\rangle)$.
	Since $R_j\to\infty$ and $d(h_jy_j,y_j)=1$, we have $h_jy_j\overset{GH}\to y'$ and $S'y'\subseteq \bar{B}_1(y')$ with respect to the above convergence.
	
	\textbf{Claim:} The set $S'y'$ is closed under multiplication. The proof is similar to Claim 4 in the proof of Proposition \ref{c1}. Let $\beta_1,\beta_2\in S'$ and let $b_{j,1},b_{j,2}\in [-1,1]$ such that
	$$(R_j^{-1}Y_j,y_j,b_{j,1}h_j,b_{j,2}h_j)\overset{GH}\longrightarrow (Y',y',\beta_1,\beta_2).$$
	If $b_{j,1}+b_{j,2}\in[-1,1]$, then clearly $\beta_1\beta_2\in S'$. If not, we write
	$$b_{j,1}+b_{j,2}=\pm 1 + o_j,$$
	where $o_j\in [-1,1]$. In a convergent subsequence, we have
	$$(R_j^{-1}Y_j,y_j,h_j,o_jh_j)\overset{GH}\longrightarrow (Y',y',h',\beta_0)$$
	with $\beta_0\in S'$. Then with respect to this convergence,
	$$\beta_1\beta_2y'=\lim_{j\to\infty} (o_jh_j)\cdot(\pm h_j) y_j=\beta_0y'\in S'y'.$$
	This proves the Claim.
	
	Together with Corollary \ref{bdd_sym_fix}, we conclude $S'y'=y'$. On the other hand, by the construction of $S_j$ and $R_j$, the limit $S'y'$ should have a point with distance $1$ to $y'$. A contradiction.
\end{proof}

\begin{rem}
	Let $z\in Hy \subseteq Gy$ as in Lemma \ref{boomer_control} and let $d=d(z,y)$. Recall that by Proposition \ref{C_tunnel}, there is a tunnel $\sigma:[0,1]\to Gy$ from $y$ to $z$ that is contained in $\overline{B_{C_0 d}}(y)$. Lemma \ref{boomer_control} shows that we can follow a specific tunnel, the one-parameter orbit, such that it is contained in $\overline{B_{C_1 d}}(y)$.
\end{rem} 

\section{Hausdorff dimension and orbit distance estimates}\label{sec_haus}

This section studies the Hausdorff dimension of the orbit $Hy$ in $(Y,y,H)\in \Omega(\widetilde{M},\langle \gamma \rangle)$ and proves Proposition C(2).

In subsection \ref{subsec_haus_one_para}, we prove distance controls on the orbit $Hy$ and show that the supremum of $\dimH(Hy)$ among all $(Y,y,H)\in \Omega(\widetilde{M},\langle \gamma \rangle)$ can be obtained (Proposition \ref{dim_sup}). In subsection \ref{subsec_pf_c2}, we relate the Hausdorff dimension of $Hy$ to a lower bound on the orbit length (Proposition \ref{orbit_length_lower}) and then complete the proof of Proposition C(2). Lastly, we have a short subsection \ref{subsec_recover} that relates Proposition C(2) to previous results on virtual abelianness \cite{Pan_esgap,Pan_cone}.

\subsection{Hausdorff dimension of one-parameter orbits}\label{subsec_haus_one_para}

We fix an element $\gamma\in \zeta_{l-1}(\mathcal{N})-\{\mathrm{id}\}$. From Proposition C(1), we know that for all $(Y,y,H)\in\Omega(\widetilde{M},\langle\gamma \rangle)$, the orbit $Hy$ is homeomorphic to $\mathbb{R}$. For each $(Y,y,H)\in\Omega(\widetilde{M},\langle\gamma \rangle)$, we choose an orbit point $z\in Hy$ with $d(z,y)=1$. The choice of such a point $z$ may not be unique since the orbit $Hy$ may cross $\partial B_1(y)$ multiple times. By Lemma \ref{boomer_control}, we always have distance control
$$d(tz,y)\le C_1$$
for all $t\in[0,1]$. For convenience, we denote
$$\Omega(\widetilde{M},\langle \gamma \rangle,1)=\{(Y,y,H,z)|(Y,y,H)\in\Omega(\widetilde{M},\langle \gamma \rangle),z\in Hy\cap \partial B_1(y)\}.$$
Given $(Y,y,H,z)\in \Omega(\widetilde{M},\langle \gamma \rangle,1)$, for each $L\in\mathbb{Z}_+$, we define
$$\mathcal{O}^L_{(Y,y,H,z)}=\{tz\ |\ t\in[0,1/L]\}\subseteq Hy.$$
If the space $(Y,y,H,z)$ is clear, we shall write $\mathcal{O}^L$ for simplicity.

\begin{lem}\label{partition_dist_estimate}
Let $C_1=C_1(\widetilde{M},\gamma)$ be the constant in Lemma \ref{boomer_control}. Then the followings hold for all $(Y,y,H,z)\in \Omega(\widetilde{M},\langle\gamma \rangle,1)$ and all $L\in\mathbb{Z}_+$:\\
%(1) $\mathrm{diam}(\mathcal{O}^L)\le C_1\cdot d(\frac{1}{L}z,y)\le C_1$;\\
(1) $\mathrm{diam}(\mathcal{O}^L)\le 2C_1^3L^{-1/n}$;\\
(2) $(L+1)\cdot\mathrm{diam}(\mathcal{O}^{L+1})^s\le L\cdot \mathrm{diam}(\mathcal{O}^{L})^s+ C_1^s$, where $s\ge 1$.  
\end{lem}

\begin{proof}
   (1) We write $r=d(\frac{1}{L}z,y)$. First note that by Lemma \ref{boomer_control}, the points $\{\frac{j}{L}z\}_{j=0}^{L}$ are all contained in $\overline{B_{C_1}}(y)$. 
   
   We claim that the points $\{\frac{j}{L}z\}_{j=0}^{L}$ are pairwise $C_1^{-1}r$-disjoint. In fact, suppose that there are $j_1<j_2$ in $\{0,1,...,L\}$ with
   $$d\left(\frac{j_1}{L}z,\frac{j_2}{L}z\right)< C_1^{-1}r.$$
   Since $\mathbb{R}z$ is represented by the orbit of a one-parameter subgroup at $y$, we have
   $$d\left(\frac{j_1-j_2}{L}z,y\right)< C_1^{-1}r.$$
   However, by Lemma \ref{boomer_control}, 
   $$r=d((1/L)z,y)\le C_1\cdot d\left(\frac{j_1-j_2}{L}z,y\right)<r.$$
   A contradiction. This verifies the claim.
    
   Next, by a standard packing argument with respect to a limit renormalized measure on $Y$, we have 
   $$L\le \left(\dfrac{C_1}{C_1^{-1}r/2}\right)^n=(2C_1^2)^{n}\cdot r^{-n}.$$
   Thus
   $$\mathrm{diam}(\mathcal{O}^L)\le C_1\cdot r\le 2 C_1^{3}\cdot L^{-1/n}.$$
   (2) It is clear that
   $$\mathrm{diam}(\mathcal{O}^{L+1})\le \mathrm{diam}(\mathcal{O}^{L})$$
   because $\mathcal{O}^{L+1} \subseteq \mathcal{O}^{L}$. Thus
   \begin{align*}
   	(L+1)\cdot\mathrm{diam}(\mathcal{O}^{L+1})^s \le& L\cdot \mathrm{diam}(\mathcal{O}^{L})^s + \mathrm{diam}(\mathcal{O}^{L+1})^s\\
   	\le& L\cdot \mathrm{diam}(\mathcal{O}^{L})^s+C_1^s,
   \end{align*}
   where the last inequality holds because $\mathrm{diam}(\mathcal{O}^{L+1})\le C_1$.
\end{proof}

Recall that for a metric space $(X,d)$, we have definition
$$\mathcal{H}^s_\delta(X)=\inf \left\{\sum_{j=1}^\infty r_j^s\ \bigg|\  X\subseteq \cup_{j=1}^\infty B_j, \text{ where each $B_j$ has diameter } r_j\le\delta\right\},$$
then $s$-dimensional Hausdorff measure and Hausdorff dimension of $X$ are defined by
$$\mathcal{H}^s(X)=\lim_{\delta\to 0} \mathcal{H}^s_\delta(X),$$
$$\dimH(X)=\inf \{s>0| \mathcal{H}^s(X)=0\}=\sup \{s>0| \mathcal{H}^s(X)=\infty\}.$$
When $X$ is compact, we can use finite covers $\{B_j\}$ instead of countable ones to define $\mathcal{H}_\delta^s(X)$.

Next, we use equal partitions of $\mathcal{O}^1$ to give an alternative way to calculate its Hausdorff dimension.

\begin{defn}
We define 
$$\mathcal{E}^s(\mathcal{O}^1)=\liminf\limits_{L\to\infty} L\cdot \mathrm{diam}(\mathcal{O}^L)^s,$$
where $L$ takes values in $\mathbb{Z}_+$.
\end{defn}

\begin{lem}\label{Hdim_equal_partition}
	Given $s\ge 1$, there is a constant $C_2=C_2(\widetilde{M},\gamma,s)>1$ such that
	$$\mathcal{H}^s(\mathcal{O}^1_{(Y,y,H,z)})\le\mathcal{E}^s(\mathcal{O}^1_{(Y,y,H,z)})\le C_2\cdot  \mathcal{H}^s(\mathcal{O}^1_{(Y,y,H,z)})$$
	holds for any $(Y,y,H,z)\in\Omega(\widetilde{M},\langle\gamma\rangle,1)$. As a consequence, 
	$$\dimH(Hy)=\inf \{s>0| \mathcal{E}^s(\mathcal{O}^1)=0\}=\sup \{s>0| \mathcal{E}^s(\mathcal{O}^1)=\infty \}.$$
\end{lem}

\begin{proof}
	We first show that $\mathcal{H}^s(\mathcal{O}^1)\le \mathcal{E}^s(\mathcal{O}^1)$, which is straightforward. Let $L\in\mathbb{Z}_+$. Note that $\mathcal{O}^L$ has diameter at most $2C_1^3L^{-1/n}$ by Lemma \ref{partition_dist_estimate}(1). For each $j=1,...,L$, we define 
	$$\mathcal{B}_j=\{tz\ |\ t\in [(j-1)/L,j/L]\}.$$
	Then $\{\mathcal{B}_j\}_{j=1}^L$ covers $\mathcal{O}^1$ and each $\mathcal{B}_j$ has diameter at most $2C_1^3 L^{-1/n}$. This shows that for $\delta=2C_1^3L^{-1/n}$, we have
	$$\mathcal{H}^s_\delta(\mathcal{O}^1) \le \sum_{j=1}^{L} \mathrm{diam}(\mathcal{B}_j)^s= L\cdot \mathrm{diam}(\mathcal{O}^L)^s.$$
	Let $L\to\infty$, then $\delta\to 0$ and we conclude that $\mathcal{H}^s(\mathcal{O}^1)\le \mathcal{E}^s(\mathcal{O}^1)$.
	
	Next, we prove that $\mathcal{E}^s(\mathcal{O}^1)\le C_2\cdot  \mathcal{H}^s(\mathcal{O}^1)$ for some constant $C_2(\widetilde{M},\gamma,s)$. Let $\delta>0$. Let $\{B_j\}_{j=1}^J$ be a cover of $\mathcal{O}^1$ such that each $B_j$ is contained in $\mathcal{O}^1$, has diameter $r_j\le \delta$, and
	$$\sum_{j=1}^J r_j^s \le 2 \cdot\mathcal{H}^s_\delta(\mathcal{O}^1).$$
	By replacing $B_j$ by its closure, without loss of generality, we can assume that all $B_j$ are closed. For each $j$, let $B'_j\subseteq \mathcal{O}^1$ be the smallest connected closed subset that contains $B_j$; in other words, we put
	$$\alpha_j=\min\{t|tz\in B_j\},\quad \beta_j=\max\{t|tz\in B_j\},$$
	and
	$$B'_j=\{tz|t\in[\alpha_j,\beta_j]\}.$$ 
	Note that by Lemma \ref{boomer_control},
	$$\mathrm{diam}(B'_j)\le C_1\cdot d(\alpha_j z,\beta_j z)\le  C_1r_j.$$
	
	We further modify the cover $\{B'_j\}_{j=1}^{J}$ to a cover $\{D'_j\}_{j\in I}$ in the following way, where the index set $I$ has cardinality at most $J$ and the adjacent $D'_j$ and $D'_{j+1}$ have exactly one common point. We define the index set
	$$I=\{j=1,...,J | \text{$B'_j$ is not contained in $B'_m$ for any $m\not= j$}\}.$$
	Clearly, $\{B'_j\}_{j\in I}$ still covers $\mathcal{O}^1$. We rearrange $\{B'_j\}_{j\in I}$ by the order of their left endpoints and then relabel the sets in $\{B'_j\}_{j\in I}$ to $\{D_j\}_{j=1}^{|I|}$. Let
	$$\alpha'_j=\min\{t|tz\in D_j\},\quad \beta'_j=\max\{t|tz\in D_j\};$$
	then $\alpha'_j< \alpha'_{j+1}$, $\beta'_j<\beta'_{j+1}$ for all $j=1,...,|I|-1$ and $\beta'_{|I|}=1$. Now we define
	$$D'_j=\{tz|t\in[\alpha'_j,\alpha'_{j+1}]\}$$
	for each $j=1,...,|I|-1$ and define the last one $D'_{|I|}=D_{|I|}$. By construction, $\{D'_j\}_{j=1}^{|I|}$ covers $\mathcal{O}^1$; moreover, each $D'_j$ is not a single point and the adjacent two share exactly one common point.

	Let $d_j=\mathrm{diam}(D'_j)$. By construction, each $D'_j$ is contained in $D_{j}$, and each $D_j$ is indeed the relabel of some unique $B'_{m}$, thus 
	$$\sum_{j=1}^{|I|} d_j^s \le \sum_{j=1}^{J}\mathrm{diam}(B'_j)^s\le C_1^s\cdot \sum_{j=1}^{J} r_j^s\le  2C_1^s \cdot\mathcal{H}^s_\delta(\mathcal{O}^1).$$
	For each $j= 1,...,|I|$, we put
	$$\tau_j=\max\{t|tz\in D'_j\}-\min\{t|tz\in D'_j\}.$$
	Because two adjacent $D'_j$ and $D'_{j+1}$ share exactly one common point, we see that $\sum_{j=1}^{|I|}\tau_j=1$. Let $j_0\in\{1,...,|I|\}$ such that
	$$\dfrac{d_{j_0}^s}{\tau_{j_0}}=\min_{j\in I} \dfrac{d_j^s}{\tau_j}=:\rho.$$
	We choose $L\in\mathbb{Z}_+$ with $1\le L\tau_{j_0}\le 2$. 
	Then
	\begin{align*}
    &L\cdot \mathrm{diam}(\mathcal{O}^L)^s \le  L\cdot \mathrm{diam}(D'_{j_0})^s 
    =L\cdot \rho \tau_{j_0}\\
    \le&\ 2\rho=2\rho\sum_{j=1}^{|I|} \tau_j \le 2\sum_{j=1}^{|I|} d_j^s \le  4C_1^s \cdot\mathcal{H}^s_\delta(\mathcal{O}^1).
	\end{align*}
    By triangle inequality,
    $$L\cdot \mathrm{diam}(\mathcal{O}^L)\ge L\cdot d((1/L)z,y)\ge d(z,y)=1.$$
    Thus 
    $$L\ge \dfrac{1}{\mathrm{diam}(\mathcal{O}^L)}\ge \dfrac{1}{\mathrm{diam}(D'_{j_0})}\ge \dfrac{1}{C_1\delta}.$$
    This means that the above chosen $L\to\infty$ as $\delta\to 0$. Let $\delta\to 0$; we conclude that
    $$\mathcal{E}^s(\mathcal{O}^1)\le 4C_1^s\cdot \mathcal{H}^s(\mathcal{O}^1).$$
\end{proof}

\begin{prop}\label{dim_sup}
	$\mathcal{D}:=\sup\{\dimH(Hy)|(Y,y,H)\in\Omega(\widetilde{M},\langle\gamma \rangle)\}$ can be obtained.
\end{prop}

\begin{proof}
	If $\mathcal{D}=1$, then $\dimH(Hy)=1$ for all $(Y,y,H)\in\Omega(\widetilde{M},\langle\gamma\rangle)$ and the result holds trivially. Below, we assume that $\mathcal{D}>1$.
	
	\textbf{Claim:} Given $s\in(1,\mathcal{D})$, there is $(Y',y',H',z')\in \Omega(\widetilde{M},\langle\gamma \rangle,1)$ such that
	$$K\cdot \mathrm{diam}(\mathcal{O}^K_{(Y',y',H',z')})^s\ge 1$$
	for all integers $K\ge 1$.
	
	Let $\{(Y_j,y_j,H_j,z_j)\}_j$ be a sequence of spaces in $\Omega(\widetilde{M},\langle\gamma \rangle,1)$ such that $$\lim_{j\to\infty} \dimH(H_jy_j)\to \mathcal{D},$$
	Let $s\in(1,\mathcal{D})$, then $s<\dimH(H_jy_j)$ for all $j$ large. By Lemma \ref{Hdim_equal_partition}, this implies that
	$$L\cdot \mathrm{diam}(\mathcal{O}^L_{(Y_j,y_j,H_j,z_j)})^s \to \infty$$
	as $L\to\infty$. For each $j\in\mathbb{N}$, we choose $L_j\in\mathbb{N}$ as the smallest integer such that $$L\cdot \mathrm{diam}(\mathcal{O}^L_{(Y_j,y_j,H_j,z_j)})^s \ge 2^j$$ for all $L\ge L_j$. By Lemma \ref{partition_dist_estimate}(2) and the choice of $L_j$, such an $L_j$ also satisfies
    $$L_j\cdot \mathrm{diam}(\mathcal{O}^{L_j}_{(Y_j,y_j,H_j,z_j)})^s\in [2^j,2^j+C_1^s].$$
    Let $r_j=d((1/L_j)z_j,y_j)$. By Lemma \ref{partition_dist_estimate}(1), $r_j\to 0$ as $j\to \infty$. 
    Passing to a subsequence if necessary, we consider
    $$(r_j^{-1}Y_j,y_j,H_j,(1/L_j)z_j)\overset{GH}\longrightarrow (Y',y',H',z')\in \Omega(\widetilde{M},\langle\gamma \rangle,1).$$
    Let $K\in \mathbb{Z}_+$. We estimate that
    \begin{align*}
    	r_j^{-1}\cdot\mathrm{diam}(\mathcal{O}^{KL_j}_{(Y_j,y_j,H_j,z_j)})\ge & \dfrac{\mathrm{diam}(\mathcal{O}^{KL_j}_{(Y_j,y_j,H_j,z_j)})}{\mathrm{diam}(\mathcal{O}^{L_j}_{(Y_j,y_j,H_j,z_j)})}\\
    	\ge & \dfrac{\left(\dfrac{2^j}{KL_j}\right)^{1/s}}{\left(\dfrac{2^j+C_1^s}{L_j}\right)^{1/s}}
    	\to \left(\dfrac{1}{K}\right)^{1/s}
    \end{align*}
    as $j\to \infty$. Note that for any $K\in \mathbb{Z}_+$, it holds the convergence
    $$(r_j^{-1}Y_j,y_j,\mathcal{O}^{KL_j}_{(Y_j,y_j,H_j,z_j)})\overset{GH}\longrightarrow (Y',y',\mathcal{O}^K_{(Y',y',H',z')}).$$
    Thus
    $$K\cdot \mathrm{diam}(\mathcal{O}^K_{(Y',y',H',z')})^s\ge K \cdot \dfrac{1}{K}=1$$
    for any integer $K\ge 1$. Hence $(Y',y',H',z')$ is the desired space for the claim.
    
    With this claim, for a sequence $s_j\in (1,D)$ with $s_j\to D$ we can choose a sequence $(Y_j,y_j,H_j,z_j)\in \Omega(\widetilde{M},\langle\gamma \rangle,1)$ such that
    $$K\cdot \mathrm{diam}(\mathcal{O}^K_{(Y_j,y_j,H_j,z_j)})^{s_j}\ge1$$
    for all integers $K\ge 1$ and all $j$. Using precompactness again, we have convergence
    $$(Y_j,y_j,H_j,z_j)\overset{GH}\longrightarrow (Y,y,H,z)\in \Omega(\widetilde{M},\langle\gamma \rangle,1).$$
    For each fixed $K\ge 1$, it is clear that
    $$K \cdot \mathrm{diam}(\mathcal{O}^K_{(Y,y,H,z)})^D=\lim\limits_{j\to\infty} K\cdot \mathrm{diam}(\mathcal{O}^K_{(Y_j,y_j,H_j,z_j)})^{s_j} \ge 1.$$
    This shows that $\dimH(Hy)\ge D$ by Lemma \ref{Hdim_equal_partition}. Hence $\dimH(Hy)= D$.
\end{proof}

\subsection{Proof of Proposition C(2)}\label{subsec_pf_c2}

We continue to use the notation $\mathcal{D}$ as in Proposition \ref{dim_sup}. 

\begin{lem}\label{uniform_partition}
	Let $s>\mathcal{D}$ and $\epsilon>0$. Then there is a constant $L_0=L_0(\epsilon,s,\widetilde{M},\gamma)$ such that for all $(Y,y,H,z)\in\Omega(\widetilde{M},\langle \gamma \rangle,1)$, there exists an integer $2\le L\le L_0$ with
	$$L\cdot \mathrm{diam}(\mathcal{O}^L_{(Y,y,H,z)})^s\le\epsilon.$$
\end{lem}

\begin{proof}
	We argue by contradiction. Suppose that for each integer $L_j=j$, there is some space $(Y_j,y_j,H_j,z_j)\in \Omega(\widetilde{M},\langle \gamma \rangle,1)$ such that
	$$L\cdot \mathrm{diam}(\mathcal{O}^L_{(Y_j,y_j,H_j,z_j)})^s > \epsilon$$
	for all $2\le L\le L_j$. After passing to a subsequence, we consider
	$$(Y_j,y_j,H_j,z_j)\overset{GH}\longrightarrow (Y',y',H',z').$$
	For any integer $L\ge 2$, we observe that
	$$L\cdot\mathrm{diam}(\mathcal{O}^L_{(Y',y',H',z')})^s=\lim\limits_{j\to\infty} L\cdot \mathrm{diam}(\mathcal{O}^L_{(Y_j,y_j,H_j,z_j)})^s\ge \epsilon.$$
	Thus $\dimH(H'y')\ge s$ by Lemma \ref{Hdim_equal_partition}, which is a contradiction to $s>\mathcal{D}$.
\end{proof}

Below we write $|\gamma|=d(\gamma \tilde{p},\tilde{p})$ for convenience. Next, we transfer Lemma \ref{uniform_partition} to a distance estimate of $\langle\gamma\rangle$-action on $\widetilde{M}$.

\begin{lem}\label{length_growth_estimate}
	Let $s>\mathcal{D}$. Then there are constants $L'=L'(s,\widetilde{M},\gamma)$ and $R=R(s,\widetilde{M},\gamma)$ such that the following holds.
	
	For any $\gamma^b\in\langle\gamma\rangle$ with $|\gamma^b|\ge R$, where $b\in\mathbb{Z}_+$, there exists an integer $2\le L\le L'$ such that
	$$|\gamma^b|\ge L^{1/s}\cdot |\gamma^{\lceil b/L \rceil}|,$$
	where $\lceil \cdot \rceil$ means the ceiling function.
\end{lem}

\begin{proof}
	We choose $L'(s,\widetilde{M},\gamma)=L_0(\frac{1}{2},s,\widetilde{M},\gamma)$, the constant in Lemma \ref{uniform_partition}. We argue by contradiction to prove the statement. Suppose that there is a sequence $b_i\to\infty$ such that
	$$|\gamma^{b_i}|\le L^{1/s}\cdot |\gamma^{\lceil b_i/L\rceil}|$$
	for all $L=2,...,L'.$ Let $r_i=|\gamma^{b_i}|\to\infty$, we consider
	$$(r_i^{-1}\widetilde{M},\tilde{p},\langle\gamma\rangle,\gamma^{b_i})\overset{GH}\longrightarrow (Y,y,H,g).$$
	It is clear that $g\in H$ satisfies $d(gy,y)=1$. We put $z=gy$. 
	
	We claim that for each integer $L\in\mathbb{Z}_+$, we have convergence 
	$$(r_i^{-1}\widetilde{M},\tilde{p},\gamma^{\lceil b_i/L \rceil}\tilde{p}) \overset{GH}\to (Y,y,(1/L)z).$$
	In fact, we can write
	$$\lceil b_i/L \rceil\cdot L=b_i+o_i,$$
	where $o_i\in \{0,1,...,L-1\}$. After passing to a subsequence, we have
	$$(r_i^{-1}\widetilde{M},\tilde{p},\gamma^{\lceil b_i/L \rceil},\gamma^{o_i})\overset{GH}\longrightarrow (Y,y,h,\delta)$$
	with $h^L=g \delta$; moreover, $\delta y=y$ because each $o_i$ is at most $L-1$. Thus we have
	$$h^L y= g\delta y= gy=z.$$
	By Lemma \ref{limit_one_para}(3), we have a one-parameter subgroup $\sigma:\mathbb{R}\to H_0$ with $\sigma(1)=g'$ and $g'y=z$. Then by Corollary \ref{orbit_power_back} and the fact that $h^L y=\sigma(1) y$, we conclude that
	$$hy=\sigma(1/L)y=(1/L)z.$$
	This proves the Claim.
	
	From the above Claim and the hypothesis, we have
	$$\mathrm{diam}(\mathcal{O}^L_{(Y,y,H,z)})\ge d((1/L)z,y)=\lim\limits_{i\to\infty} \dfrac{d(\gamma^{\lceil b_i/L \rceil}\tilde{p},\tilde{p})}{d(\gamma^{b_i}\tilde{p},\tilde{p})}\ge \left(\dfrac{1}{L}\right)^{1/s}$$
	for all $L\in\{2,...,L'\}$. On the other hand, by the choice $L'=L_0(\frac{1}{2},s,\widetilde{M},\gamma)$ and Lemma \ref{uniform_partition}, 
	$$\mathrm{diam}(\mathcal{O}^L_{(Y,y,H,z)})\le \left(\dfrac{1}{2L}\right)^{1/s}$$
	for some $L\in\{2,...,L'\}$. A contradiction.
\end{proof}

Then we use Lemma \ref{length_growth_estimate} repeatedly to derive a lower bound for $|\gamma^b|$.

\begin{prop}\label{orbit_length_lower}
	Let $s>\mathcal{D}$. Then there is a constant $C_3=C_3(s,\widetilde{M},\gamma)$ such that 
	$$|\gamma^b|\ge C_3\cdot b^{1/s}$$
	for all $b\in \mathbb{Z}_+$ large.
\end{prop}	

\begin{proof}
	Let $s>\mathcal{D}$ and let $P_0$ be a large integer such that
	$$|\gamma^b|\ge R(s,\widetilde{M},\gamma)$$
	holds for all $b\ge P_0$, where $R(s,\widetilde{M},\gamma)$ is the constant in Lemma \ref{length_growth_estimate}.
	
	Let $b>P_0$. By Lemma \ref{length_growth_estimate}, there is some integer $L_1\in\{2,...,L'\}$ such that
	$$|\gamma^b|\ge L_1^{1/s}\cdot |\gamma^{\lceil b/L_1 \rceil}|.$$
	If $\lceil b/L_1 \rceil\le P_0$, then we stop here. Otherwise, we apply Lemma \ref{length_growth_estimate} again to find some integer $L_2\in \{2,...,L'\}$ such that
	$$|\gamma^b|\ge L_1^{1/s}\cdot |\gamma^{\lceil b/L_1 \rceil}|\ge (L_1L_2)^{1/s}\cdot |\gamma^{\lceil\lceil b/L_1 \rceil/L_2\rceil}|.$$
	Repeating this process, we eventually obtain
	$$|\gamma^b|\ge \left( \prod_j L_j \right)^{1/s} \cdot |\gamma^{\lceil...\lceil b/L_1\rceil/L_2.../L_k\rceil}|\ge \left( \prod_j L_j \right)^{1/s} \cdot r_0,$$
	where $\lceil...\lceil b/L_1\rceil/L_2.../L_k\rceil\le P_0$ and $r_0=\min_{m\in \mathbb{Z}_+}|\gamma^m|$. Note that
	$$\dfrac{b}{\prod_j L_j}\le \lceil...\lceil b/L_1\rceil/L_2.../L_k\rceil \le P_0.$$
	It follows that
	$$|\gamma^b|\ge \left(\dfrac{b}{P_0}\right)^{1/s}\cdot r_0=C_3\cdot b^{1/s},$$
	where $C_3=r_0/P_0^{1/s}$.
\end{proof}	

As indicated in the introduction, Proposition C(2) follows immediately from Proposition \ref{orbit_length_lower} and Corollary \ref{orbit_length_upper}.

\begin{proof}[Proof of Proposition C(2)]
	Let $s>\mathcal{D}$. By Proposition \ref{orbit_length_lower}, we have a lower bound
	$$|\gamma^b|\ge C_3\cdot b^{1/s}$$
	for all $b\in \mathbb{Z}_+$ large. On the other hand, by Corollary \ref{orbit_length_upper}, we have an upper bound
	$$|\gamma^b|\le C_4\cdot b^{1/l}$$
	for all $b\in\mathbb{Z}_+.$ Combining these two inequalities together, we obtain
	$$C_3\cdot b^{1/s}\le C_4\cdot b^{1/l}$$
	for all $b$ large. We conclude that $s\ge l$. Recall that $s>\mathcal{D}$ is arbitrary, thus $\mathcal{D}\ge l$. Lastly, by Proposition \ref{dim_sup}, there exists $(Y,y,H)\in\Omega(\widetilde{M},\langle \gamma \rangle)$ such that
	$$\dimH(Hy)=\mathcal{D}\ge l.$$
	This completes the proof.
\end{proof}

\subsection{Relations to previous results on virtual abelianness}\label{subsec_recover}

As indicated in the introduction, Proposition C(2) immediately recovers the main result on metric cones and virtual abelianness in \cite{Pan_cone}. Recall that an open manifold with $\mathrm{Ric}\ge 0$ is \textit{conic at infinity}, if every asymptotic cone $(Y,y)$ is a metric cone with vertex $y$.

\begin{cor}\label{cor_cone}
	Let $(M,p)$ be an open $n$-manifold with $\mathrm{Ric}\ge 0$ and $E(M,p)\not=1/2$. If its Riemannian universal cover is conic at infinity, then $\pi_1(M)$ is virtually abelian.
\end{cor}

\begin{proof}
	 We shall show that 
	$$\dimH(Ly)=1$$
	for any $(Y,y)\in \Omega(\widetilde{M})$ and any closed $\mathbb{R}$-subgroup $L$ of $\mathrm{Isom}(Y)$. Then by Proposition C(2), we have 
	$$\step(\pi_1(M))\le 1,$$
	that is, $\pi_1(M)$ is virtually abelian (see Definition \ref{def_virnilstep}). 
	
	The verification of $\dimH(Ly)=1$ is standard given the metric cone structure. We give some details below for readers' convenience. Let $(Y,y)$ be any asymptotic cone of $\widetilde{M}$. By assumption, $Y$ is a metric cone with vertex $y$. By Cheeger-Colding splitting theorem, $(Y,y)$ splits isometrically as
	$$(\mathbb{R}^k \times C(Z),(0,z)),$$
	where $k\in\mathbb{N}\cap [0,n]$, $C(Z)$ is a metric cone without lines, and $z$ is the unique vertex of $C(Z)$. Moreover, its isometry group also splits as a product
	$$\mathrm{Isom}(Y)=\mathrm{Isom}(\mathbb{R}^k) \times \mathrm{Isom} (C(Z)).$$
	Because $z$ is the unique vertex of $C(Z)$, any isometry of $C(Z)$ must fix $z$. As a consequence, 
	$$g\cdot y=g\cdot (0,z)\in \mathbb{R}^k \times \{z\}$$
	for any isometry $g$ of $Y$. In particular, we have the orbit
	$$Ly \subseteq \mathbb{R}^k \times \{z\}$$
	for any closed $\mathbb{R}$-subgroup $L$ of $\mathrm{Isom}(Y)$. Thus we can view $Ly$ as a $C^1$-curve in the Euclidean factor $\mathbb{R}^k \times \{z\}$. Consequently, $Ly$ must have Hausdorff dimension $1$.
\end{proof}

Proposition C(2) also extends the main result on small escape rate and virtual abelianness in \cite{Pan_esgap}. To explain this, we prove a proposition below, which is based on the results in this paper and \cite{Pan_esgap}. We continue to use the notation $\mathcal{D}$ from Proposition \ref{dim_sup}.

\begin{prop}\label{cor_small_escape}
	Let $M$ be an open $n$-manifold with $\mathrm{Ric}\ge 0$ and $E(M,p)\le \epsilon$, where $\epsilon>0$ is a small number. Then 
	$$\mathcal{D}\le 1+\delta(\epsilon|n),$$
	where $\delta(\epsilon|n)\to 0$ as $\epsilon\to 0$. 
\end{prop}

\begin{proof}
	Let $\delta>0$. Suppose that $\mathcal{D}=1+\delta$. We set $s=1+\delta/2$. Following the proof of Proposition \ref{dim_sup}, we can find a space $(Y,y,H,z)\in \Omega(\widetilde{M},\langle \gamma \rangle,1)$ such that
	$$K\cdot \mathrm{diam}(\mathcal{O}^K)^s\ge 1$$
	for any integer $K\ge 1$. On the other hand, using the small escape rate condition, it follows from \cite[Theorem 0.1 and Lemma 4.6]{Pan_esgap} that
	$$d_{GH}((Y,y,H,z),(\mathbb{R}^k\times X,(0,x),\mathbb{R},(1,x)))\le \Psi(\epsilon|n),$$
	where $\times$ means a metric product and the group $\mathbb{R}$ acts as translations in $\mathbb{R}^k \times X$. For a fixed $\delta>0$ and $s=1+\delta/2$, we can choose a large integer $K$ such that
	$$2^s K^{1-s} \le 1/2.$$
	When $\epsilon$ is so small that $\Psi(\epsilon|n)\le 1/K$, we have
	$$1 \le K\cdot \mathrm{diam}(\mathcal{O}^K_{(Y,y,H,z)})^s \le K\cdot (\frac{1}{K}+\Psi)^s\le 2^s K^{1-s} \le 1/2.$$
	This clear contradiction shows that $\delta \to 0$ as $\epsilon \to 0$.
\end{proof}

Combining Proposition \ref{cor_small_escape} and Proposition C(2), it is clear that $\step(\pi_1(M))\ge 2$ implies $E(M,p)>\epsilon(n)$ for some universal constant $\epsilon(n)$, which is the main result in \cite{Pan_esgap}.
	
\appendix

\section{An application of the slice theorem}	

In this appendix, we prove Proposition \ref{G0_by_orbit} below, which is used in the proof of Lemmas \ref{tunnel_by_isotropy} and \ref{limit_one_para}(1).

\begin{prop}\label{G0_by_orbit}
	Let $Y$ be a Ricci limit space and let $G$ be a closed nilpotent subgroup of $\mathrm{Isom}(Y)$. Then there exists a point $z\in Y$ such that $G$ acts freely at $z$.
\end{prop}

The proof of Proposition \ref{G0_by_orbit} depends on a slice theorem by Palais \cite{Palais}. For other applications of the slice theorem to Ricci limit spaces, see \cite{PanWang,Wang}.
 
For the convenience of readers, we recall some basic notions. Let $Y$ be a completely regular topological space and let $G$ be a Lie group acting properly and effectively on $Y$ by homeomorphisms. Let $y\in Y$ and let $H=\mathrm{Iso}(G,y)$ be the isotropy subgroup at $y$. We say a subset $S\subseteq Y$ is $H$-invariant, if $H\cdot S=S$. Given an $H$-invariant subset $S$, we define a space $G\times_H S$ as a quotient $(G\times S)/\sim$, where the equivalence relation on $G\times S$ is given by
$$(g,s)\sim (gh,h^{-1}s)$$
for all $g\in G$, $h\in H$ and $s\in S$. We use the $[g,s]$ to write elements in this quotient space $G\times_H S$, where $g\in G$ and $s\in S$. The space $G\times_H S$ has a natural $G$-action by
$$g'\cdot [g,s]=[g'g,s],$$
where $g'\in G$ and $[g,s]\in G\times_H S$.

\begin{defn}
	We call $S\subseteq Y$ a slice at a point $y\in Y$, if the followings hold:\\
	(1) $S$ is $H$-invariant, where $H=\mathrm{Iso}(G,y)$;\\
	(2) $y\in S$ and $G\cdot S$ is an open neighborhood of $y$;\\
	(3) the map
	$$\psi: G\times_H S \to G\cdot S,\quad [g,s]\mapsto g\cdot s $$
	is an $G$-equivariant homeomorphism.
\end{defn}

The slice theorem by Palais \cite{Palais} guarantees the existence of a slice when $G$ is a Lie group. We formulate its statement in our context as follows. Because the isometry group of a Ricci limit space is always Lie \cite{CN}, the slice theorem applies.

\begin{thm}\cite{Palais}\label{slice}
	 Let $Y$ be a Ricci limit space and let $G$ be a closed subgroup of $\mathrm{Isom}(Y)$. Then for every $y\in Y$, there is a slice $S$ at $y$. 
\end{thm}

Note that the isotropy subgroup at $y$ and at $gy$ are related by a conjugation
$$\mathrm{Iso}(G,gy)=g \cdot \mathrm{Iso}(G,y) \cdot g^{-1}.$$
Hence each orbit $Gy$ corresponds to a conjugacy class of isotropy subgroups. Let 
$$\mathcal{I}(G)=\{H\le G| H=\mathrm{Iso}(G,y) \text{ for some } y\in Y \}$$
be the set of all isotropy subgroups. For $H\in\mathcal{I}(G)$, we denote by $[H]$ the conjugacy class of $H$. We define a partial order on the set of conjugacy classes: $[H]\le [K]$ if $H$ is conjugate to a subgroup of $K$ by some element $g\in G$.

\begin{defn}
	Let $H\in\mathcal{I}(G)$. We say that $[H]$ is principal, if $[K]\le [H]$ implies $[K]=[H]$, where $K\in \mathcal{I}(G)$.
\end{defn}

\begin{lem}\label{exist_prin}
	Let $Y$ be a Ricci limit space and let $G$ be a closed subgroup of $\mathrm{Isom}(Y)$. Then there is $H\in \mathcal{I}(G)$ such that $[H]$ is principle.
\end{lem}

\begin{proof}
	It suffices to show that any descending sequence
	$$[H_1]\ge [H_2] \ge ...\ge[H_i]\ge...$$
	must stabilize. Since $H_i\in\mathcal{I}(G)$, it is a compact subgroup of $G$; in particular, each $H_i$ is Lie and has only finitely many components. If $H_{i+1}$ is conjugate to a proper subgroup of $H_i$, then either $\dim(H_{i+1})<\dim(H_i)$, or $\dim(H_{i+1})=\dim(H_i)$ but $H_{i+1}$ has strictly less connected components than $H_i$. It follows that the sequence stabilizes.
\end{proof}

For a nilpotent isometric $G$-action on a Ricci limit space $Y$, we show that a principal $[H]$ must be trivial.

\begin{thm}\label{prin_orb}
	Let $Y$ be a Ricci limit space and let $G$ be a closed nilpotent subgroup of $\mathrm{Isom}(Y)$. Suppose $H\in\mathcal{I}(G)$ such that $[H]$ is principal, then $H=\{ \mathrm{id} \}$.
\end{thm}

To prove Theorem \ref{prin_orb}, beside the slice theorem, we need Lemma \ref{nil_G0_commute_cpt} and the following characterization of the identity map from \cite[Lemma 2.1]{PanRong}.

\begin{lem}\cite{PanRong}\label{fix_id}
	Let $X$ be a Ricci limit space and let $g$ be an isometry of $X$. Suppose that $g$ fixes every point in an open ball in $X$, then $g$ is the identity map on $X$.
\end{lem} 

\begin{proof}[Proof of Theorem \ref{prin_orb}]
	The proof is similar to that of the manifold case.
	
	Let $y\in Y$ such that $\mathrm{Iso}(G,y)=H$. By Theorem \ref{slice}, there is a slice $S$ at $y$. Because the open neighborhood $G\cdot S$ is $G$-equivariantly homeomorphic to $G\times_H S$, we study this local model $G\times_H S$ to understand $G\cdot S$.
	
	\textbf{Claim:} In $G\times_H S$, $\mathrm{Iso}(G,[g,s])$ is conjugate to $H$ by $g$. By definition, if $g'[g,s]=[g,s]$, then there is some $h\in H$ such that 
	$$(g'g,s)=(gh,h^{-1}s)$$
	in $G\times S$. Hence $h$ fixes $s$ and $g'=ghg^{-1}$. As a result, $\mathrm{Iso}(G,[g,s])$ must be conjugate to a subgroup of $H$ by $g$. Because $[H]$ is principal, the claim follows.
	
	Note that the claim in particular implies $\mathrm{Iso}(G,[e,s])=H$. Let $h\in H$. Then $h$ fixes every point in $S$. For any $g_0\in G_0$ and any $s\in S$, by Lemma \ref{nil_G0_commute_cpt} and compactness of $H$, we have
	$$hg_0 s=g_0 hs=g_0 s.$$
	Thus $h$ fixes every point in $G_0\cdot S$, which is an open neighborhood of $y$. Applying Lemma \ref{fix_id}, we conclude $h=\mathrm{id}$. Hence $H=\{ \mathrm{id} \}$.
\end{proof}

\begin{proof}[Proof of Proposition \ref{G0_by_orbit}]
	By Lemma \ref{exist_prin} and Theorem \ref{prin_orb}, there is a point $z\in Y$ such that the isotropy subgroup $\mathrm{Iso}(G,z)$ is trivial. The result follows.
\end{proof}

\end{document}